\documentclass[12pt, reqno]{amsart}

\usepackage{amsmath, amsthm, amssymb}
\usepackage{enumerate}
\usepackage[margin=1.0in]{geometry}
\usepackage{xcolor}
\definecolor{cite}{rgb}{0.30,0.60,1.00}
\definecolor{url}{rgb}{0.00,0.00,0.80}
\definecolor{link}{rgb}{0.40,0.10,0.20}
\usepackage[pdfusetitle,colorlinks,linkcolor=link,urlcolor=url,citecolor=cite,pagebackref,breaklinks]{hyperref}
\usepackage{graphicx}
\usepackage{subcaption}
\usepackage{cleveref}
\usepackage{mathdots}
\usepackage{tikz-cd}
\usepackage{comment}
\usepackage[all]{xy}
\usepackage{mathtools}
\usepackage{physics}
\usepackage{float}

\newtheorem{theorem}{Theorem}[section]

\newtheorem{proposition}[theorem]{Proposition}
\newtheorem{lemma}[theorem]{Lemma}

\newtheorem{corollary}[theorem]{Corollary}

\theoremstyle{definition}
\newtheorem{definition}[theorem]{Definition}
\theoremstyle{definition}
\newtheorem{remark}[theorem]{Remark}
\theoremstyle{definition}
\newtheorem{example}[theorem]{Example}
\theoremstyle{definition}
\newtheorem{assumption}[theorem]{Assumption}

\newcommand{\Hom}{\mathrm{Hom}}

\newcommand{\conjugate}[1]{\overline{#1}}

\newcommand{\sizeof}[1]{\left|#1\right|}

\renewcommand{\innerproduct}[2]{\left\langle #1,#2\right\rangle}

\newcommand{\fieldCharacter}{\psi}
\newcommand{\centralCharacter}[1]{\omega_{#1}}
\newcommand{\Ind}[3]{\mathrm{Ind}_{#1}^{#2}\left(#3\right)}

\newcommand{\Whittaker}{\mathcal{W}}

\newcommand{\besselFunction}{\mathcal{J}}
\newcommand{\SO}{\mathrm{SO}}
\newcommand{\GSO}{\mathrm{GSO}}
\newcommand{\Sp}{\mathrm{Sp}}
\newcommand{\ad}{\mathrm{ad}}

\newcommand{\SL}{\mathrm{SL}}
\newcommand{\K}{\mathrm{K}}
\newcommand{\A}{\mathrm{AS}_{\psi^{-1}}}
\newcommand{\WC}{\mathrm{WC}}
\newcommand{\Ql}{\overline{\mathbb{Q}}_\ell}

\newcommand{\opp}{\mathrm{opp}}
\newcommand{\Bun}{\mathrm{Bun}_{\mathcal{G}}}

\newcommand{\Hk}{\mathrm{Hk}}
\newcommand{\pr}{\mathrm{pr}}
\newcommand{\Gr}{\mathrm{Gr}}
\newcommand{\IC}{\mathrm{IC}}
\newcommand{\ext}{\mathrm{ext}}

\newcommand{\grpIndex}[2]{\left[#1:#2\right]}

\newcommand{\IdentityMatrix}[1]{I_{#1}}
\newcommand{\diag}{\mathrm{diag}}

\newcommand{\GL}{\mathrm{GL}}
\newcommand{\Sodd}{\mathrm{SO}_{2n+1}}
\newcommand{\Spn}{\mathrm{Sp}_{2n}}
\newcommand{\GSpn}{\mathrm{GSp}_{2n}}
\newcommand{\GU}{\mathrm{GU}_{n}}

\newcommand{\Esix}{\mathrm{E}_6}
\newcommand{\Esev}{\mathrm{E}_7}
\newcommand{\Gsix}{\mathrm{GE}_6}
\newcommand{\Gsev}{\mathrm{GE}_7}

\newcommand{\finiteField}{\mathbb{F}}

\newcommand{\whittakerVector}[1]{v_{#1, \fieldCharacter}}

\newcommand{\opposite}[1]{{#1}_{\mathrm{-}}}
\newcommand{\specialWeylElement}[1]{\dot{w}_{#1}}

\newcommand\restr[2]{{% we make the whole thing an ordinary symbol
  \left.\kern-\nulldelimiterspace % automatically resize the bar with \right
  #1 % the function
  \vphantom{\big|} % pretend it's a little taller at normal size
  \right|_{#2} % this is the delimiter
  }}

\hypersetup{pdfauthor={Robert Cass, Miao (Pam) Gu, Elad Zelingher},
	pdfsubject={Representation theory},
	pdfkeywords={Bessel functions, Kloosterman sheaves}}

    \author{Robert Cass}
\address{Mathematical Sciences Department, Claremont McKenna College, 850 Columbia Avenue, Claremont, CA 91711 USA}
\email{robert.cass@claremontmckenna.edu}

\author{Miao (Pam) Gu}
\address{Department of Mathematics, Aarhus University, Ny Munkegade 118, DK-8000 Aarhus, Denmark}
\email{pmgu@math.au.dk}

\author{Elad Zelingher}
\address{Department of Mathematics, University of Southern California, Los Angeles, CA 90089-2532, USA}
\email{zelinghe@usc.edu}


\title[Kloosterman sheaves and Bessel functions]{Kloosterman sheaves and Bessel functions \\ for generic principal series of finite groups}

\begin{document}

\begin{abstract}
Let $G$ be a quasi-split reductive group over a finite field and let $\hat{G}$ be the Langlands dual group. Assuming the derived subgroup of $G$ is almost simple, we use techniques from the geometric Langlands program to relate special values of Bessel functions for generic principal series representations of $G$ to the trace of Frobenius acting on Kloosterman sheaves of Heinloth--Ng\^o--Yun for $\hat{G}$. We also give explicit examples in essentially all possible cases, including exceptional groups.
\end{abstract}

\maketitle

\tableofcontents

\section{Introduction}

A Whittaker model is an explicit realization of a representation of a $p$-adic group. Whittaker models play a prominent role in many aspects of the Langlands program, such as constructions of integral representations of $L$-functions \cite{JacquetPiatetskiShapiroShalika1983, Shahidi1984, Shahidi1990} and indexing of local $L$-packets \cite[Section 9, item (2)]{GGP2012}.  Representations admitting Whittaker models are called \emph{generic}. 
Whittaker models can also be defined for finite groups of Lie type.
More precisely, let $\mathbb{F}$ be a finite field of order $q$ and characteristic $p$. Let $G / \mathbb{F}$ be a connected, quasi-split reductive group whose derived subgroup is almost simple. Fix a Borel subgroup $B = TU$ and let $\psi \colon U(\mathbb{F}) \rightarrow \mathbb{C}^{\times}$ be a generic character; the particular $\psi$ we work with factors through an algebraic map $\phi \colon U \rightarrow \mathbb{G}_a$. Given an irreducible generic representation $\tau$ of $G(\mathbb{F})$, we may define a special matrix coefficient called a \emph{Bessel function}, denoted by $\besselFunction_{\tau, \psi}$.

In \cite{Zelingher2023}, the third author showed that for $G = \GL_n$, the value of $\besselFunction_{\tau, \psi}$ at an element of the form $\operatorname{anti-diag}(cI_k, I_{n-k})$ is, up to normalization, the trace of Frobenius acting on a geometric stalk of an exterior power of an exotic Kloosterman sheaf defined by Deligne--Katz \cite{Deligne1977,Katz1988}. Recently, Heinloth--Ng\^o--Yun \cite{HNY} defined Kloosterman sheaves for quasi-split reductive groups, using ideas from the geometric Langlands program. In this work we generalize \cite{Zelingher2023} from $\GL_n$ to arbitrary $G$ as above, in the case of principal series.

To state our main result, fix an isomorphism $\mathbb{C}  \cong \overline{\mathbb{Q}}_\ell$ for some prime $\ell \neq p$, and a character $\chi \colon T(\mathbb{F}) \rightarrow \overline{\mathbb{Q}}_\ell^\times$.
We consider the unique irreducible subrepresentation $\tau \subset \Ind{B_-(\mathbb{F})}{G(\mathbb{F})}{\chi}$ admitting a $\psi$-Whittaker vector.
For the generalization of $\operatorname{anti-diag}(cI_k, I_{n-k})$, let $P \subset G$ be the maximal parabolic subgroup associated to a simple root $\alpha_i$ whose adjoint fundamental coweight $\omega_i^\vee$ is minuscule and defined over $\mathbb{F}$. Let $L \subset P$ be the Levi factor and let $\dot{w}_P \in G(\mathbb{F})$ be the lift (determined by the conventions in Section \ref{sec:Crystal}) of the minimal length element in the maximal left coset of the Weyl group of $P$ in that of $G$.

\begin{theorem} \label{mainthm} Suppose we have an isomorphism $T \cong Z(G) \times T_{\ad}$. For $z \in Z(L)(\mathbb{F})$, use this isomorphism to write $z = z_0 z_1$ where $z_0 \in Z(G)(\mathbb{F})$ and $z_1 \in Z(L_{\ad})(\mathbb{F})$. Then
$$\textnormal{Kl}_{\hat{G}}^{\omega_i^\vee}(\phi, \chi^{-1}; \alpha_i(z))  = \overline{\chi(z_0)} (-1)^{\dim G/P} q^{\tfrac{\dim G/P}{2}} \besselFunction_{\tau, \fieldCharacter}\left(z \specialWeylElement{P}\right).$$ 
\end{theorem}

To explain the left side of Theorem \ref{mainthm}, note that the isomorphism $T \cong Z(G) \times T_{\ad}$ provides a lift of $\omega_i^\vee$ to $T \subset G$. The left side is then the trace of Frobenius acting on a geometric stalk above $\alpha_i(z) = \alpha_i(z_1) \in \mathbb{F}^{\times}$ of the Kloosterman sheaf of Heinloth--Ng\^o--Yun \cite[Remark 2.5]{HNY} for the irreducible representation of the Langlands dual group $\hat{G}$ of highest weight $\omega_i^{\vee}$. The power of $q$ in the right side of Theorem \ref{mainthm} reflects the fact that we normalize our Kloosterman sheaves to have unitary Frobenius eigenvalues; see Definition \ref{KDef} and Remark \ref{rem:Norm} for further discussion. Our normalization also includes some minor differences consistent with the work of Lam--Templier \cite{LamTemplier2024} and Xu--Zhu \cite{XZ22} as discussed in Remark \ref{NormRemark}. We refer to Section \ref{Sec:Qsplit} for the precise definition in the quasi-split case.

For split groups, a minuscule fundamental coweight exists when $G$ is of type A, B, C, D, $\Esix$ or $\Esev$. For quasi-split groups, a rational minuscule fundamental coweight exists for certain involutions of the Dynkin diagram when $G$ is of type $\mathrm{A}_n$ for $n$ odd, or of type $\mathrm{D}_n$ for arbitrary $n$. We write down explicitly the elements $z \specialWeylElement{P}$ in all cases in Section \ref{Sec:ExamplesSplit}; see in particular Tables \ref{tab:minuscule_coweights_AC} and \ref{tab:minuscule_coweights_D} for classical groups and Section \ref{sec:E} for type E. Even if there is no isomorphism $T \cong Z(G) \times T_{\ad}$, one can easily reduce to this case by choosing an embedding as explained in Remark \ref{notTorus}, e.g.~$\SL_n \rightarrow \GL_n$ and further examples in Sections \ref{Sec:Qsplit} and \ref{Sec:ExamplesSplit}.

Bessel functions were first introduced by Gelfand--Graev \cite{GelfandGraev1962} for $\SL_2\left(\finiteField\right)$, and later  for general linear groups by S.~I.~Gelfand in \cite{Gelfand1970}.
Bessel functions encode a great deal of representation-theoretic information. For example, given an irreducible cuspidal representation of a finite general linear group, one can give an explicit matrix realization of this representation using its Bessel function \cite{AlperinJames1995}. Bessel functions also appear naturally in formulas for Rankin--Selberg gamma factors of irreducible generic representations \cite{Roditty2010, Nien2014, LiuZhang2022, LiuZhang2022b, SoudryZelingher2023, LiuHazeltine2024}.

Explicit formulas for values of Bessel functions for $\GL_n\left(\finiteField\right)$ were computed for $n=2$ \cite{PiatetskiShapiro1983, Carter1992}, for $n=3$ \cite{Chang1976, HelversenPasotto1982, Carter1992} and for some cases of $n=4$ \cite{Gotsis1997, DeriziotisGotsis1998, ShinodaTulunay2005}. These formulas are often difficult and do not suggest a pattern.
A formula for values of Bessel functions of $\GL_n\left(\finiteField\right)$ at elements of the form $\operatorname{anti-diag}(c,\IdentityMatrix{n-1})$ was given in \cite{CurtisShinoda2004, Tulunay2004, Nien2017} in terms of exotic Kloosterman sums, which was then generalized in \cite{Zelingher2023} as discussed above.
For groups $G$ other than $\SL_2\left(\finiteField\right)$ and $\GL_n\left(\finiteField\right)$ almost no explicit computation of values of Bessel functions has been done before our work. We are only aware of results of Rainbolt \cite{Rainbolt2002, Rainbolt2008} for unitary groups of small rank and of Breeding-Allison--Rainbolt \cite{BreedingAllisonRainbolt2019} for $\operatorname{GSp}_4\left(\finiteField\right)$. 

Our proof of Theorem \ref{mainthm} is entirely different from the proof for $\GL_n\left(\finiteField\right)$ by the third author in \cite{Zelingher2023}.
In the latter work, the relation between Bessel functions and Kloosterman sheaves appears coincidental, as it is established using a third connection with gamma factors. In this work, we first find an explicit expression (Theorem \ref{thmbessel}) for $\besselFunction_{\tau, \fieldCharacter}$ at special elements in terms of a parabolic geometric crystal of $G$ (Definition \ref{def:Crystal}) introduced by Berenstein--Kazhdan \cite{BK07}. This brings algebraic geometry into the picture, and allows us to construct an object $\WC_G^P(\phi, \chi)$ (Definition \ref{defWC}) in the derived category of \'etale sheaves on $Z(L)$ which encodes special values of $\besselFunction_{\tau, \fieldCharacter}$ under the function-sheaf dictionary. The identification in the minuscule case with a Kloosterman sheaf in Sections \ref{sec:K} and \ref{Sec:proof} closely follows work of Lam--Templier \cite{LamTemplier2024} on $D$-modules of geometric crystals. Here we deal with some additional technicalities since we do not wish to limit ourselves to simple adjoint groups. We show the following, again assuming we have fixed an isomorphism $T \cong Z(G) \times T_{\ad}$, which induces an isomorphism $Z(L) \cong Z(G) \times Z(L_\ad)$.

\begin{theorem}Let $\K_\chi$ be the character sheaf on $T$ associated to $\chi \colon T(\mathbb{F}) \rightarrow \overline{\mathbb{Q}}_\ell^\times$.
    Viewing $\alpha_i$ as an isomorphism $Z(L_{\ad}) \rightarrow \mathbb{G}_m$, there exists an isomorphism of $\ell$-adic sheaves on $Z(L) \colon$
    $$\WC_G^P(\phi, \chi)  \cong \restr{\K_{\chi}}{Z(G)} \boxtimes \alpha_i^*\textnormal{Kl}_{\hat{G}}^{\omega_i^\vee}(\phi, \chi^{-1}).$$
\end{theorem}
Combined with \cite{HNY}, it follows that $\WC_G^P(\phi, \chi)$ is a local system, pure of weight 0, and can viewed as parametrizing a family of Hecke eigenvalues for an automorphic function.
Our method has the advantage that it works uniformly for all quasi-split groups, including exceptional groups, and the relation between Bessel functions and Kloosterman sheaves is more apparent. The disadvantage is that we are only able to use it for principal series representations at this time. We aim to generalize our results to non-minuscule parabolics and other irreducible generic representations in the future. 

\subsection*{Acknowledgments} During part of the preparation of this article, R.C.~was supported by the National Science Foundation grant DMS 1840234. R.C.~also thanks Claremont McKenna College and the University of Michigan for excellent working conditions. M.P.G.~is thankful for partial support provided by an AMS-Simons Travel Grant. M.P.G.~also thanks Stephen DeBacker, Shenghao Li, and Griffin Wang for helpful discussions and Bao Ch\^au Ng\^o for encouragement. E.Z.~would like to thank Zhiwei Yun for multiple interesting discussions about Kloosterman sheaves and Bessel functions, for sending him unpublished notes, and for his encouragement. E.Z.~would also like to thank the University of Michigan for excellent working conditions and, in particular, is grateful to Charlotte Chan for her continuous support. 
Gemini, Claude, and ChatGPT were used for proofreading and LaTeX formatting. Additionally, they were used to assist the authors in constructing the examples in Section \ref{Sec:ExamplesSplit} as discussed there (see also \cite{CassGuZelingher2026GitHub}), and to help identify a minor inconsistency between the setup \cite{HNY} in \cite{LamTemplier2024} which is noted around \eqref{level-str} and discussed in Remarks \ref{LTrem}, \ref{NormRemark}. The manuscript was written entirely by the authors, and they take full responsibility for its content. 

\section{Preliminaries on reductive groups} \label{sect:prelim}
Let $G$ be a split (connected) reductive algebraic group over any base scheme. We further assume the derived subgroup of $G$ is almost simple. We continue to assume that $G$ is split until Section \ref{Sec:Qsplit}, where we explain how to deduce the quasi-split case by Galois equivariance. Let $T \subset B \subset G$ be a split maximal torus and Borel subgroup. Let $X^*(T)$ and $X_*(T)$ be the groups of characters and cocharacters. Let $R = R^+ \sqcup R^-$ be the roots, and let $\Delta$ and $\Delta^\vee$ be the simple roots and coroots determined by $B$. Let $Z(G) \subset G$ be the center of $G$. There is an exact sequence $1 \rightarrow Z(G) \rightarrow G \rightarrow G_{\ad} \rightarrow 1$, where $G_{\ad}$ is the adjoint quotient of $G$. Let $T_{\ad}$ be the image of $T$ in $G_{\ad}$; this is a split maximal torus. 
\begin{lemma} \label{lem-split}
    If $Z(G)$ is a torus,  there is a non-canonical isomorphism 
    \begin{equation} \label{eq-split}
        T \cong Z(G) \times T_{\ad}.
    \end{equation}
\end{lemma}

\begin{proof}
    The exact sequence $1 \rightarrow Z(G) \rightarrow T \rightarrow T_{\ad} \rightarrow 1$ corresponds to an exact sequence of finitely generated abelian groups
    $$0 \rightarrow X^*(T_\ad) \rightarrow X^*(T) \rightarrow X^*(Z(G)) \rightarrow 0.$$ Splittings $T \rightarrow Z(G)$ are in bijection with splittings $X^*(Z(G)) \rightarrow X^*(T)$. Since $Z(G)$ is a torus, $X^*(Z(G))$ is free, so at least one splitting exists.
\end{proof}

The quotient $T \rightarrow T_{\ad}$ induces a bijection between $\Delta$ and $\Delta_{\ad}$, and similarly for simple coroots. For every subset $I \subset \Delta$, there is a parabolic subgroup $P \subset G$ generated by $B$ and the root subgroups $U_{-\alpha}$ for $\alpha \in I$. We call $P$ a standard parabolic subgroup, and its Levi factor  $L \subset P$ is called a standard Levi subgroup. 

Taking quotients and preimages under the map $G \rightarrow  G_{\ad}$ induces a bijection between standard parabolics and Levi factors in these groups. In what follows, we use the subscript $\ad$ to denote the corresponding objects for $G_{\ad}$ under this bijection. For example, $\Delta_{\ad} \subset X^*(T_{\ad})$ and $L_{\ad} = L/Z(G)$. We emphasize that $L_{\ad}$ is in general different from the adjoint quotient $L/Z(L)$, although there is always a surjection with central kernel from $L_{\ad}$ to the latter.

\begin{corollary}\label{cor:splitting-of-center-of-levi-part} Let $L \subset G$ be a standard Levi subgroup. If $Z(G)$ is a torus, then a choice of isomorphism as in \eqref{eq-split} determines an isomorphism $$Z(L) \cong Z(G) \times Z(L_{\ad}).$$
\end{corollary}

\begin{proof}
    Regardless of whether $Z(G)$ is a torus, there is always an exact sequence $1 \rightarrow Z(G) \rightarrow Z(L) \rightarrow Z(L_{\ad}) \rightarrow 1$. The composition of the inclusion $Z(L) \rightarrow T$ with  \eqref{eq-split} provides a splitting $Z(L) \rightarrow Z(G)$.
\end{proof}

We now recall the definition of minuscule parabolics. First, recall that the monoid of dominant cocharacters $X_*(T)^+ \subset X_*(T)$ consists of those $\mu$ such that $\langle \alpha_i, \mu \rangle \geq 0$ for all $\alpha_i \in \Delta$. There is a partial order $\leq$ on $X_*(T)^+$ such that $\mu_1 \leq \mu_2 $ if and only if $\mu_2 - \mu_1$ is a non-negative integer combination of simple coroots.

The quotient $T \rightarrow T_{\ad}$ induces a natural map $X_*(T) \rightarrow X_*(T_{\ad})$ which preserves the pairing with elements of $\Delta = \Delta_{\ad}$, and hence there is a map $X_*(T)^+ \rightarrow X_*(T_{\ad})^+$. Each connected component of the poset $X_*(T)^+$ maps isomorphically onto a connected component in $X_*(T_{\ad})^+$. Furthermore, if $Z(G)$ is a torus the maps $X_*(T) \rightarrow X_*(T_{\ad})$ and $X_*(T)^+ \rightarrow X_*(T_{\ad})^+$ are surjective by \eqref{eq-split}.

In what follows we let $\hat{G}$ be the Langlands dual group of $G$, i.e.~ the split reductive group over $\mathbb{C}$ with root datum dual to that of $G$. Let $N_G(T) \subset G$ be the normalizer of $T$, and let $W = N_G(T)/T $ be the Weyl group, which also identifies with the Weyl groups of $G_{\ad}$ and $\hat{G}$.

\begin{definition}
    Let $\mu \in X_*(T)^+$ be a cocharacter whose image $\mu_{\ad} \in X_*(T_{\ad})^+$ is nonzero. We say that $\mu$ is \emph{minuscule} if it satisfies any of the following equivalent conditions (cf.~\cite[Lemma 1.1]{NP01}).

\begin{enumerate}
    \item We have $| \langle \alpha,\mu \rangle | \leq 1$ for every $\alpha \in R$.
    \item The cocharacter $\mu$ is minimal with respect to the partial order on $X_*(T)^+$.
    \item The weights of the irreducible algebraic representation $V_\mu$ of $\hat{G}$ of highest weight $\mu$ belong to the single Weyl orbit $W \cdot \mu \subset X_*(T) = X^*(\hat{T})$.
\end{enumerate}
\end{definition}

Recall that $\Delta_{\ad}$ is a free $\mathbb{Z}$-basis of $X^*(T_{\ad})$. For each $\alpha_i \in \Delta$ let $\omega_i^\vee \in X_*(T_{\ad})$ be the corresponding fundamental coweight, i.e.~$\langle \alpha_j, \omega_i^\vee \rangle = \delta_{ij}$. The $\omega_i^\vee$ form a free $\mathbb{Z}$-basis of $X_*(T_{\ad})$. Every minuscule cocharacter of $T_{\ad}$ is a fundamental coweight, but the converse is not true in general. The minuscule cocharacters of $T_{\ad}$ are in one-to-one correspondence with nonzero elements of the (finite) quotient $X_*(T_\ad) / \mathbb{Z}[\Delta_\ad^\vee]$, in the sense that each equivalence class contains a unique minuscule representative.

\begin{definition}
    We say that a standard parabolic $P \subset G$ is \emph{minuscule} if it  is associated to a subset of simple roots of the form $I = \Delta \setminus \{\alpha_i\}$ for some $\alpha_i \in \Delta = \Delta_\ad$, and if the fundamental coweight $\omega_i^\vee \in X_*(T_{\ad})$ associated to $\alpha_i$ is minuscule. 
\end{definition}

When $I = \Delta \setminus \{\alpha_i\}$ the center of the Levi $L$ admits a natural $\mathbb{G}_m$ factor as below, on which we will later construct a Kloosterman sheaf in the minuscule case.

\begin{lemma} \label{GmLevi}
Let $L \subset G$ be the Levi factor associated to a subset of the form $I = \Delta \setminus \{\alpha_i\}$. If $Z(G)$ is a torus, then a choice of isomorphism as in \eqref{eq-split} determines an isomorphism $$Z(L) \cong Z(G) \times \mathbb{G}_m.$$
\end{lemma}

\begin{proof}
    This follows from Corollary \ref{cor:splitting-of-center-of-levi-part} and the fact that the composition $Z(L_\ad) \rightarrow T_{\ad} \xrightarrow{\alpha_i} \mathbb{G}_m$ is an isomorphism.
\end{proof}

The following result will be used to handle reductive groups with arbitrary centers.

\begin{lemma} \label{lemm-G'} Let $(G,T)$ be a split reductive group and a split maximal torus. Then there exists another such pair $(G', T')$ along with an embedding $(G,T) \rightarrow (G', T')$ such that:
\begin{enumerate}
    \item $Z(G')$ is a torus.
    \item $Z(G)$ remains central in $G'$.
    \item The inclusion $T \rightarrow T'$ splits (non-canonically).
    \item The induced map $G_\ad \rightarrow G'_\ad$ is an isomorphism.
\end{enumerate}
\end{lemma}

\begin{proof}
    Choose a split torus $S$ and an embedding $\varphi \colon Z(G) \rightarrow S$ (for example, we can take $S=T$). Let $G'$ be the pushout of $Z(G) \rightarrow G$ along $Z(G) \rightarrow S$, that is, $G'$ is the quotient of $S \times G$ by the central normal subgroup $N = \{(\varphi(z^{-1}), z) \: \mid \: z \in Z(G)\}$. Then $S \times G \rightarrow G'$ is a surjection of split reductive groups with central kernel, so $T' = (S \times T)/N$ is a maximal torus and $Z(G') = (S \times Z(G))/N$. We take for the inclusion $(G,T) \rightarrow (G', T')$ the one induced by the obvious maps. The natural map $S \rightarrow Z(G')$ is an isomorphism, so $Z(G')$ is a torus. Since $T$ is a torus the inclusion $T \rightarrow T'$ also splits as in Lemma \ref{lem-split}. The remaining claims follow immediately.
\end{proof}

\begin{example}
    Let $G = \SL_n$ with the diagonal torus $T$. Then $Z(G) = \mu_n$, and we may take $S = \mathbb{G}_m$ with the canonical embedding $\varphi \colon \mu_n \rightarrow \mathbb{G}_m$. The map $\mathbb{G}_m \times \SL_n \rightarrow \GL_n$ given by the diagonal central embedding on the first factor and the usual inclusion on the second factor induces an isomorphism $G' \cong \GL_n$ with diagonal torus $T'$.
\end{example}

\section{Geometric Crystals} \label{sec:Crystal} We continue to use the notation from Section \ref{sect:prelim} concerning the split reductive algebraic group $G$. Additionally, we let $U_- \subset B_-$ be the opposite unipotent radical and Borel. 

For each root we choose a trivialization $u_\alpha \colon  \mathbb{G}_a  \rightarrow  U_\alpha $ of the corresponding root group, and for each $w \in W$ we choose a representative $\dot{w} \in N_G(T)$. We require that these choices satisfy the compatibilities outlined in \cite[\S 2.2]{LamTemplier2024}, which amounts to choosing a homomorphism $\SL_2 \rightarrow G$ for each pair of a positive root and its negative. In fact, we can make these choices for a split model of $G_{\text{sc}}$, the simply connected cover of the derived subgroup of $G$, defined over $\mathbb{Z}$. Then our choices of root trivializations (which just amount to signs) and $\dot{w}$ will be defined over $\mathbb{Z}$ and compatible with passage to $G_{\ad}$ and  $G'$ as in Lemma \ref{lemm-G'}. The $\dot{w}$ are uniquely determined by the lifts of the simple reflections by taking a product of lifts in a reduced expression.
Additionally, we define the following additive character.
\begin{itemize}
    \item Let $\phi \colon U \rightarrow \mathbb{G}_a$ be the homomorphism given by the sum of the simple root groups with respect to our chosen trivializations.
\end{itemize}

Our next task is to define the geometric crystals associated to $G$ as in \cite{BK07}.

\begin{lemma} \label{lem-Uw}
    Let $U(w) = U \cap \dot{w} U_- \dot{w}^{-1}$ with the reduced subscheme structure. Then the multiplication map
$$T \times U(w) \times U \rightarrow T U \dot{w} U, \quad (t, u_1, u_2) \mapsto t u_1 \dot{w} u_2$$ is an isomorphism, where the target is a locally closed subscheme of $G$. 
\end{lemma}

\begin{proof}
    Note that
$$\dot{w}^{-1} T U \dot{w} U = T \dot{w}^{-1} U \dot{w} U = T (\dot{w}^{-1} U \dot{w} \cap U_-) U = T \dot{w}^{-1} U(w) \dot{w} U.$$ The result follows since the multiplication map $T \times U_- \times U \rightarrow G$ is an open embedding.
\end{proof}

Since $T$ normalizes $U$ and $U(w)$, we have $T U \dot{w} U = U T \dot{w} U$, so the analogue of Lemma \ref{lem-Uw} also holds if we swap the order of $T$ and $U(w)$. Let
$$B_-^w := U \dot{w} U \cap B_-$$ with the reduced subscheme structure. Then the multiplication map induces an isomorphism from $T \times B_-^w$ onto $U T \dot{w} U \cap B_-$.

\begin{lemma} \label{Bad}
    The quotient $G \rightarrow G_{\ad}$ restricts to an isomorphism $B_-^w \rightarrow (B_{\ad})_-^w$.
\end{lemma}

\begin{proof} The identity $\dot{w}^{-1} U \dot{w} U = \dot{w}^{-1} U(w) \dot{w} U$ implies that $U \dot{w} U$ maps isomorphically onto its image in $G_{\ad}$. Since $B_{-}$ is the preimage of $(B_{\ad})_-$ in $G$,  the restriction of this isomorphism over $(B_{\ad})_-^w$ provides the necessary inverse.
\end{proof}

Fix a subset $I \subset \Delta$ and the associated standard parabolic subgroup $P \subset G$ with Levi factor $L$. Let $W_P$ be the subgroup of $W$ generated by the simple reflections in $I$. Let $w_0^P \in W_P$ be the longest element, and let $w_P = w_0^P w_0$. This is the minimal length representative of the longest element in $W_P \backslash W$.

\begin{definition} \label{def:Crystal}
    The \emph{geometric crystal} associated to $(G, P)$ is the (reduced) intersection
    $$X := U Z(L) \dot{w}_P U \cap B_-.$$
\end{definition}

By the discussion following Lemma \ref{lem-Uw}, the multiplication map $Z(L) \times B_-^{w_P} \rightarrow X$ is an isomorphism. Let $X_{\ad}$ denote the geometric crystal associated to $(G_{\ad}, P_{\ad})$. The quotient $G \rightarrow G_{\ad}$ restricts to a map $X \rightarrow X_{\ad}$ which factors as the product of the quotient $Z(L) \rightarrow Z(L_{\ad})$ and the isomorphism in Lemma \ref{Bad}.

\begin{definition} \label{def-Xmaps}
    We define the following algebraic maps associated to the crystal $X$.

    \begin{enumerate}
        \item The \emph{highest weight map} is
$$\pi \colon X \rightarrow Z(L), \quad u_1 t \dot{w}_P u_2 \mapsto t.$$
    \item The \emph{weight map} is
$$\gamma \colon X \rightarrow T, \quad x \mapsto x \text{ mod } U_- \in B_-/U_- = T.$$
    \item The \emph{decoration} is
$$f \colon X \rightarrow \mathbb{A}^1, \quad u_1 t \dot{w}_P u_2 \mapsto \phi(u_1) + \phi(u_2).$$
    \end{enumerate}
\end{definition}

We note that the decoration factors through the quotient $X \rightarrow X_{\ad}$.
Recall that if $Z(G)$ is a torus, a choice of isomorphism $T \cong Z(G) \times T_{\ad}$ as in \eqref{eq-split} determines an isomorphism $Z(L) \cong Z(G) \times Z(L_{\ad})$ as in Corollary \ref{cor:splitting-of-center-of-levi-part}.
The latter induces an isomorphism $X \cong Z(G) \times X_{\ad}$. 

\begin{corollary} \label{cor-weightSplit}
If $Z(G)$ is a torus, a choice of isomorphism as in \eqref{eq-split} determines a lift $\gamma_{\ad}^T \colon X_{\ad} \rightarrow T$ of the weight map $\gamma_{\ad}\colon X_{\ad} \rightarrow T_\ad$.
\end{corollary}

\begin{proof}
We have $X_{\ad} = Z(L_\ad) \times (B_{\ad})_-^{w_P}$. For the first factor, we lift the inclusion $Z(L_\ad) \rightarrow T_{\ad}$ to $T$ along our chosen inclusion $T_{\ad} \rightarrow T$. For the second factor, by Lemma \ref{Bad} we may lift $(B_{\ad})_-^{w_P} \rightarrow T_{\ad}$ to $T$ by using the corresponding map $B_-^{w_P} \rightarrow T$ for $G$.
\end{proof}

When $Z(G)$ is a torus, we can summarize the relationship between $X$ and $X_{\ad}$ in the following diagram, which is commutative except for the lower triangle in the top square.

\begin{equation}\label{Xmaps}\xymatrix{
T \ar[r] & T_{\ad} & \\ 
X \ar[r] \ar[d]^-\pi \ar[u]^-{\gamma} & X_{\ad} \ar@{.>}_{\gamma_{\ad}^T}[lu] \ar[d]^-{\pi_{\ad}} \ar[u]_-{\gamma_\ad} \ar[r]^{f_\ad} & \mathbb{A}^1 \\
Z(L) \ar[r] & Z(L_{\ad}) &}\end{equation}
Here the three parallel horizontal maps are the projections which forget the factor $Z(G)$, and a choice of isomorphism as in \eqref{eq-split} determines sections of all of these. In what follows we are essentially reduced to studying the adjoint crystal $X_{\ad}$ with the slightly more general (lifted) weight map $\gamma_{\ad}^T \colon X_{\ad} \rightarrow T$.

When $Z(G)$ is not a torus a similar diagram exists, except there is no lift $\gamma_{\ad}^T \colon X_{\ad} \rightarrow T$ in general. Instead, when $Z(G)$ is not a torus we will choose an embedding $(G, T) \rightarrow (G',T')$ as in Lemma \ref{lemm-G'}, which induces an embedding of crystals
\begin{equation} \label{XtoX'} X = Z(L) \times B_-^{w_P} \rightarrow Z(L') \times B_-^{w_P} \cong Z(L') \times (B')_-^{w_{P'}} = X'.\end{equation}

\section{Generic representations and Bessel functions}\label{sec:generic-representations-and-bessel-functions}
Let $\mathbb{F}$ be a finite field of order $q$, and let $p$ be the characteristic of $\mathbb{F}$. We keep the notation surrounding $G$ and its geometric crystal $X$ from Sections \ref{sect:prelim} and \ref{sec:Crystal}, and we work over the base $\mathbb{F}$. Additionally, we fix the following data.
\begin{itemize}
    \item Let $\psi \colon \mathbb{F}_p \rightarrow \mathbb{C}^\times$ be a nontrivial additive character. By abuse of notation, we also denote by $\psi \colon \mathbb{F} \rightarrow \mathbb{C}^\times$ the composition of the trace homomorphism $\mathbb{F} \rightarrow \mathbb{F}_p$ with $\psi \colon \mathbb{F}_p \rightarrow \mathbb{C}^\times$.

    \item Let $\chi \colon T(\mathbb{F}) \rightarrow \mathbb{C}^\times$ be a multiplicative character.

\end{itemize}

For the rest of Section \ref{sec:generic-representations-and-bessel-functions}, we will only need to consider the $\mathbb{F}$-points of varieties and so we will use the shorthand $G = G(\mathbb{F})$, $X = X(\mathbb{F})$, $\gamma = \gamma(\mathbb{F}) \colon X(\mathbb{F}) \rightarrow T(\mathbb{F})$, etc. Under this convention, $\psi \circ \phi$ (see the start of Section \ref{sec:Crystal}) is a generic character $U \rightarrow \mathbb{C}^\times$, i.e., it is nontrivial on every simple root subgroup. In fact, all of the results in this section are valid for any generic character of $U$, but in later sections we will always work with $\psi \circ \phi$.

Given $\chi$, we may consider the parabolically induced complex representation $I_{\chi} = \Ind{\opposite{B}}{G}{\chi}$ of the finite group $G$. The irreducible subrepresentations of $I_{\chi}$ for varying $\chi$ are called \emph{principal series representations}. The vector space underlying $I_{\chi}$ is equipped with the following inner product, which is invariant under the $G$-action:
$$\innerproduct{f_1}{f_2} = \sum_{g \in \opposite{B} \backslash G} f_1\left(g\right) \conjugate{f_2\left(g\right)}.$$

A nonzero vector $v \in I_{\chi}$ such that \begin{equation} \label{eq-psi} u \cdot v = \psi\left(\phi \left(u\right)\right) v\end{equation} for all $u \in U$ is called a \emph{$\fieldCharacter$-Whittaker vector}. 
In general, an irreducible representation of $G$ (not necessarily a principal series) admitting a $\fieldCharacter$-Whittaker vector is called \emph{generic}. The uniqueness of $\fieldCharacter$-Whittaker vectors follows from multiplicity-freeness of the Gelfand–Graev representations of $G$, see e.g.~\cite[Theorem 12.3.4]{DJ20}. In the case of principal series representations this is easy to prove and has an explicit formula, as follows.

\begin{lemma} \label{lemwhit}
    There exists a unique (up to scalar multiplication) $\fieldCharacter$-Whittaker vector $v \in I_{\chi}$. Explicitly, such a vector $v$ is supported on the double coset $\opposite{B} U$ and defined there by the formula
\begin{equation} \label{explicitpsi} v\left(t u' u\right) = \chi\left(t\right) \fieldCharacter\left(\phi\left(u\right)\right),\end{equation}
for all $t \in T$, $u' \in \opposite{U}$ and $u \in U$.
\end{lemma}

\begin{proof}
    By Mackey's formula and Frobenius reciprocity,
    $$\Hom_U(\psi \circ \phi, \Ind{\opposite{B}}{G}{\chi}) = \bigoplus_{w \in W} \Hom_{wB_-w^{-1} \cap U}(\psi \circ \phi, \chi_w),$$
where $\chi_w(u) = \chi(w^{-1}uw)$ for all $u \in wB_-w^{-1} \cap U$. Since $w^{-1}uw \in U_-$ then $\chi_w$ is in fact the trivial character. 
If $w$ is nontrivial then $wB_-w^{-1} \cap U$ contains at least one simple root subgroup, so by the definition of $\psi \circ \phi$ the corresponding summand above vanishes. If $w$ is trivial then $wB_-w^{-1} \cap U$ is also trivial, so the corresponding summand above has dimension $1$. To complete the proof it suffices to observe that the formula in the statement of the lemma defines a valid vector $v$ satisfying the requirements.
\end{proof}

The $\fieldCharacter$-Whittaker vector defined by \eqref{explicitpsi} necessarily belongs to a unique irreducible subrepresentation $\tau \subset I_{\chi}$; we henceforth denote this vector by $\whittakerVector{\tau}$.

\begin{definition}
     The \emph{Bessel function} attached to the principal series representation $\tau$ with respect to the character $\fieldCharacter$ is defined by the following formula. For $g \in G$,
$$\besselFunction_{\tau, \fieldCharacter}\left(g\right) = \frac{\innerproduct{\tau\left(g\right) \whittakerVector{\tau}}{\whittakerVector{\tau}}}{\innerproduct{\whittakerVector{\tau}}{\whittakerVector{\tau}}}.$$
\end{definition}

It is evident that the Bessel function $\besselFunction_{\tau, \fieldCharacter}$ does not depend on the particular $\fieldCharacter$-Whittaker vector we used, nor on the choice of inner product on $\tau$.

\begin{remark}
	Since the center $Z\left(G\right)$ acts on vectors of $\tau$ by its central character $\centralCharacter{\tau}$, we have the identity
	$$\besselFunction_{\tau, \fieldCharacter}\left(z g\right) = \centralCharacter{\tau}\left(z\right) \besselFunction_{\tau, \fieldCharacter}\left(g\right)$$ for any $g \in G$ and $z \in Z\left(G\right)$.
\end{remark}

Let $I \subset \Delta$ with associated parabolic $P = LU_P$, and geometric crystal $X$ with its associated maps $\pi$, $\gamma$, and $f$ from Definition \ref{def-Xmaps}.
\begin{lemma} \label{lem-quot}
    We have
$$\frac{\sizeof{\left(\specialWeylElement{P}^{-1} U \specialWeylElement{P}\right) \cap U}}{\sizeof{U}} = q^{-\dim G/P}.$$
\end{lemma}

\begin{proof}
    Recall that $w_P = w_0^P w_0$. The element $w_0^P$ sends precisely those $\alpha \in R^+$ generated by $I$ to $R^-$, and then $w_0$ sends these back to $R^+$. Thus, the value above is  $q^{-c}$ where $c$ is the number of $\alpha \in R^+$ not in the subset generated by $I$. Clearly $c= \dim U_P = \dim G/P$.
\end{proof}

We have the following alternative formula for  $\besselFunction_{\tau, \fieldCharacter}$ at the special elements $z \specialWeylElement{P}$.
\begin{theorem}\label{thmbessel}
	For any $z \in Z\left(L\right)$, we have $$\besselFunction_{\tau, \fieldCharacter}\left(z \specialWeylElement{P}\right) = q^{-\dim G/P} \sum_{\substack{x \in X\\
	\pi\left(x\right) = z}} \chi\left(\gamma\left(x\right)\right) \fieldCharacter^{-1}\left(f\left(x\right)\right).$$
\end{theorem}

\begin{proof}

We have that $$\innerproduct{\whittakerVector{\tau}}{\whittakerVector{\tau}} = \sum_{x \in \opposite{B} \backslash G} \whittakerVector{\tau}\left(x\right) \conjugate{\whittakerVector{\tau}\left(x\right)} = \sum_{u \in U} \fieldCharacter\left(\phi\left(u\right)\right) \conjugate{\fieldCharacter\left(\phi\left(u\right)\right)} = \sizeof{U},$$
and that
$$\innerproduct{\tau\left(g\right) \whittakerVector{\tau}}{\whittakerVector{\tau}} = \sum_{x \in \opposite{B} \backslash G} \whittakerVector{\tau}\left(xg\right) \conjugate{\whittakerVector{\tau}\left(x\right)} = \sum_{u \in U} \whittakerVector{\tau}\left(u g\right) \conjugate{\fieldCharacter\left(\phi\left(u\right)\right)}.$$
Hence $$\besselFunction_{\tau, \fieldCharacter}\left(g\right) = \frac{1}{\sizeof{U}} \sum_{u \in U} v_{\tau, \fieldCharacter}\left(ug\right) \fieldCharacter^{-1}\left(\phi \left(u\right)\right).$$

	We are interested in all the $u \in U$ such that $u z \specialWeylElement{P} \in \opposite{B} U$. This is the same as all the pairs $u_1, u_2 \in U$ such that $u_1 z  \specialWeylElement{P} u_2 \in \opposite{B}$, i.e., all the $u_1, u_2 \in U$ such that $u_1 z  \specialWeylElement{P} u_2 \in X$. 
    Suppose that $x \coloneqq u_1 z  \specialWeylElement{P} u_2 \in X$. Then we have $$x = u_1 z  \specialWeylElement{P} u_2 = t u'$$ for some $t \in T$ and $u' \in \opposite{U}$. By definition, $t = \gamma\left(x\right)$ and $$v_{\tau, \fieldCharacter}\left(u_1 z \specialWeylElement{P}\right) \fieldCharacter^{-1}\left(\phi\left(u_1\right)\right) = v_{\tau, \fieldCharacter}\left(t u' u_2^{-1}\right) \fieldCharacter^{-1}\left(\phi\left(u_1\right)\right) = \chi\left(\gamma\left(x\right)\right) \fieldCharacter^{-1}\left(f\left(x\right)\right).$$
	Hence, $$\besselFunction_{\tau, \fieldCharacter}\left(z \specialWeylElement{P}\right) = \frac{1}{\sizeof{U}} \sum_{u_1, u_2 \in U} \sum_{\substack{x \in X\\
	u_1 z \specialWeylElement{P} u_2 = x}} \chi\left(\gamma\left(x\right)\right) \fieldCharacter^{-1}\left(f\left(x\right)\right).$$

	Notice that $u_1, u_2 \in U$ and $u'_1, u'_2 \in U$ are such that $$u_1 z \specialWeylElement{P} u_2 = u'_1 z \specialWeylElement{P} u'_2$$ if and only if $$\specialWeylElement{P}^{-1} z^{-1} \left(u'_1\right)^{-1} u_1 z \specialWeylElement{P} = u'_2 u_2^{-1}.$$ Hence every element $x \in X$ with $\pi\left(x\right) = z$ appears in the sum $\sizeof{\left(\specialWeylElement{P}^{-1} U \specialWeylElement{P}\right) \cap U}$ times, so Lemma \ref{lem-quot} completes the proof.
\end{proof}

\section{Weighted character sheaves} \label{sec-WC}
Fix a prime $\ell \neq p$ and an isomorphism $\mathbb{C} \cong \Ql$, and let $\sqrt{q} \in  \Ql$ be the image of the positive square root. Then associated to the homomorphisms chosen at the start of Section \ref{sec:generic-representations-and-bessel-functions}, we have the following $\ell$-adic local systems.

\begin{itemize}
\item Let $\A$ be the Artin--Schreier sheaf on $\mathbb{A}^1$ associated to $\psi^{-1} \colon \mathbb{F} = \mathbb{A}^1(\mathbb{F}) \rightarrow \Ql^\times$.
\item Let $\K_\chi$ be the multiplicative Kummer sheaf on $T$ associated to $\chi \colon T(\mathbb{F}) \rightarrow \Ql^\times$.
\item Let  $\K_{\chi}^{Z(G)}$ be the restriction of $\K_\chi$ to $Z(G)$.
\end{itemize}

Both  $\A$ and $\K_\chi$ are examples of character sheaves. For a commutative algebraic group $H$ over $\mathbb{F}$ with a homomorphism $\alpha \colon H(\mathbb{F}) \rightarrow \Ql^\times$, the associated character sheaf $\mathcal{L}_\alpha$ is a local system on $H$ such that $\text{Tr}(\text{Frob}_h, (\mathcal{L}_\alpha)_{\overline{h}}) = \alpha(h)$ for all $h \in H(\mathbb{F})$. Here $\text{Frob}_h$ is the geometric Frobenius at $h$ and $(\mathcal{L}_\alpha)_{\overline{h}}$ is a geometric stalk. Additionally, if $m \colon H \times H \rightarrow H$ is the group map we have $m^* \mathcal{L}_\alpha \cong \mathcal{L}_\alpha \boxtimes \mathcal{L}_\alpha$. In general, for a scheme $Z$ of finite type over $\mathbb{F}$ and an object $\mathcal{A} \in D_c^b(Z, \Ql)$, Grothendieck's function-sheaf dictionary \cite{functionsheaf} associates the function $\text{Tr}_{\mathcal{A}} \colon Z(\mathbb{F}) \rightarrow \Ql$ where $\text{Tr}_{\mathcal{A}}(z) = \sum_n (-1)^n \text{Tr}(\text{Frob}_z, \mathcal{H}^n(\mathcal{A})_{\overline{z}})$.

\begin{definition} \label{defWC} Let $d = \dim G/P$. With respect to the above data and the maps defined in Definition \ref{def-Xmaps}, the \emph{weighted character sheaf} of $X$ is
    $$\WC_G^P(\phi, \chi) : = R\pi_!(\gamma^* \K_\chi \otimes f^*\A)(\tfrac{d}{2}) [d] \in D^b_{c}(Z(L), \Ql).$$
\end{definition}

Despite the name, it is not yet clear that $\WC_G^P(\phi, \chi)$ is supported in cohomological degree zero. The Tate twist $\Ql(\tfrac{d}{2})$ is chosen so that $\WC_G^P(\phi, \chi)$ is pure of weight zero in the adjoint case.
Next, we show that taking the trace of Frobenius on the weighted character sheaf recovers special values of Bessel functions.

\begin{lemma} \label{Trace-lem}
    For any $z \in Z(L)(\mathbb{F})$, we have
    $$\textnormal{Tr}_{{\WC_G^P}(\phi, \chi)}(z) =  (-1)^{d} q^{\tfrac{d}{2}}\besselFunction_{\tau, \fieldCharacter}\left(z \specialWeylElement{P}\right).$$
\end{lemma}

\begin{proof}
   For a morphism of $\mathbb{F}$-schemes $g \colon Z \rightarrow Z'$, the derived sheaf operations $\otimes^{L}$, $g^*$, and $Rg_!$ correspond under Grothendieck's function-sheaf dictionary to pointwise multiplication, pullback, and summing over fibers. Moreover, the geometric Frobenius acts on the stalk of the Tate twist $\Ql(\tfrac{n}{2})$ by $q^{-\tfrac{n}{2}}$ for any $n \in \mathbb{Z}$. Thus, the result follows from Theorem \ref{thmbessel}.
\end{proof}

When $Z(G)$ is a torus, we now isolate a factor of $\WC_G^P(\phi, \chi)$ supported on $Z(L_{\ad})$.

\begin{lemma} \label{lem-WCad}
    Suppose that $Z(G)$ is a torus, and that we have chosen an isomorphism as in \eqref{eq-split}. If $\gamma_{\ad}^T \colon X_{\ad} \rightarrow T$ is the induced lift as in Corollary \ref{cor-weightSplit}, then with respect to the induced decomposition $Z(L) \cong Z(G) \times Z(L_{\ad})$ there is an isomorphism
     $$\WC_G^P(\phi, \chi)  = \K_{\chi}^{Z(G)} \boxtimes R\pi_{\ad, !} ((\gamma_{\ad}^T)^*\K_{\chi} \otimes f_{\ad}^* \A) (\tfrac{d}{2}) [d].$$
\end{lemma}

\begin{proof} We refer to \eqref{Xmaps} for a reminder on the maps used in the proof.
    It follows from Lemma \ref{Bad}  that $X_{\ad} = Z(L_{\ad}) \times B_-^{w_P}$, so that with respect to the decomposition $X = Z(G) \times X_{\ad}$ we have
    $$\gamma^* \K_{\chi} = \K_{\chi}^{Z(G)} \boxtimes (\gamma_{\ad}^T)^* \K_{\chi}.$$ Additionally, $f^* \A = \Ql \boxtimes f^*_{\ad} \A.$ The result then follows by applying the Kunn\"eth formula for $!$-pushforward along $\pi = \text{id}_{Z(G)} \times \pi_{\ad}$.
\end{proof}

By Lemma \ref{lem-WCad}, if $Z(G)$ is a torus we are reduced to studying the adjoint crystal $X_{\ad}$ with its highest weight map $\pi_{\ad}$, decoration $f_{\ad}$, and the more general weight map $\gamma_{\ad}^T \colon X_{\ad} \rightarrow T$. If $Z(G)$ is (possibly) not a torus, we have the following result.

\begin{lemma} \label{lem-WC'}
    Suppose that we have chosen an embedding $(G, T) \rightarrow (G', T')$ and a splitting $s \colon T' \rightarrow T$ as in Lemma \ref{lemm-G'}. Let $\chi' \colon T'(\mathbb{F}) \rightarrow \Ql^\times$ be the corresponding extension of $\chi$ to $T'(\mathbb{F}) = \ker(s)(\mathbb{F}) \times T(\mathbb{F})$ which is trivial on the first factor. Then, with respect to the embedding $Z(L) \rightarrow Z(L')$, there is an isomorphism
    $$\WC_G^P(\phi, \chi) \cong \restr{\WC_{G'}^{P'}(\phi, \chi')}{Z(L)}.$$
\end{lemma}

\begin{proof}
    Similarly to Lemma \ref{lem-WCad}, this is a straightforward diagram chase, using the embedding of crystals $X \rightarrow X'$ as in \eqref{XtoX'} and that the decoration $f \colon X \rightarrow \mathbb{A}^1$ factors through $X'$.
\end{proof}

\section{Kloosterman sheaves} \label{sec:K}

\subsection{Affine Weyl groups} Let $W_{\ext} :=X_*(T) \rtimes W$ be the Iwahori--Weyl group of $G$. For $\lambda \in X_*(T)$ we denote the corresponding translation element in $W_{\ext}$ by $\tau^{\lambda}$. We will match the sign convention in \cite{LamTemplier2024, dCHL}, so that $\tau^{\lambda}$ acts on $X_*(T) \otimes \mathbb{R}$ by $x \mapsto x + \lambda$.\footnote{We warn the reader that the opposite sign choice $x \mapsto x - \lambda$ is also common in the literature.} The resulting action of $W_{\ext}$ on affine roots is then
$$\tau^{\lambda} w (\alpha + n) = w\alpha - \langle w\alpha, \lambda \rangle + n,$$
where $\alpha$ is a root and $n \in \mathbb{Z}$. This agrees with \cite[(6.27.1)]{LamTemplier2024}, taking into account the fact that $\tau^{\lambda} w = w 
\tau^{w^{-1}(\lambda)}$. 

Using $\tau$ to also denote a formal variable, the translation element $\tau^{\lambda}$ admits a canonical lift to $G(\mathbb{F}(\!(\tau)\!))$, which we also denote by $\tau^{\lambda}$. Explicitly, for $\tau^\lambda w \in W_{\ext}$ we have the lift $\tau^{\lambda} \dot{w} = \lambda(\tau^{-1}) \dot{w} \in G(\mathbb{F}(\!(\tau)\!))$,
where $\tau^{-1}$ is used so that under our sign convention, $\tau^{\lambda} \dot{w} U_{\alpha +n} (\tau^{\lambda} \dot{w})^{-1} = U_{\tau^{\lambda} \dot{w}(\alpha +n)}$ (see \cite[\S 3.2]{dCHL}).

\subsection{Affine Grassmannians} Let $\Gr = G(\mathbb{F}(\!(\tau)\!))/ G(\mathbb{F}[\![\tau]\!])$ be the affine Grassmannian of $G$. This is represented by an ind-projective scheme over $\mathbb{F}$ (see e.g.~\cite{ZhuAffine} for more details). Identifying $G(\mathbb{F}[\![\tau]\!])$ with its associated positive loop group (an affine group scheme over $\mathbb{F}$), there is a left translation action of $G(\mathbb{F}[\![\tau]\!])$ on $\Gr$. 

By the Cartan decomposition, the left $G(\mathbb{F}[\![\tau]\!])$-orbits in $\Gr$ are represented by the points $\tau^{-\mu}$ for $\mu \in X_*(T)^+$. The orbit $G(\mathbb{F}[\![\tau]\!]) \cdot \tau^{-\mu} \cdot \Gr$ is denoted $\Gr_{\mu}$, and it is a smooth, quasi-projective $\mathbb{F}$-scheme. We emphasize that by our sign choice above, $\Gr_{\mu}$ is the $G(\mathbb{F}[\![\tau]\!])$-orbit through $\mu(\tau)$. If $\mu$ is minuscule, then $\Gr_{\mu}$ agrees with the $G$-orbit of $\tau^{-\mu}$, and the stabilizer in $G$ is the parabolic subgroup $P_{\mu}$ generated by $T$ and the root subgroups $U_\alpha$ such that $\langle \alpha, \mu \rangle \leq 0$. Thus, if $\mu$ is minuscule then $\Gr_\mu \cong G/P_\mu$.

The closure relations among the $\Gr_{\mu}$ are equivalent to the partial order $\leq$ in Section \ref{sec:Crystal}, so that in particular, if $\mu$ is minuscule then $\Gr_{\mu}$ is closed. For any $\mu \in X_*(T)^+$, the orbit $\Gr_{\mu}$ contains all the points $\tau^{-w \mu}$ for $w \in W$. Additionally, the connected components of $\Gr$ are indexed by the discrete abelian group $\Omega := X_*(T) / \mathbb{Z} \Delta^\vee = \pi_1(G).$ Here we choose this bijection so that $\Gr_\mu$ lies in the connected component indexed by the class of $\mu$ in $\Omega$.

Let $G[\tau^{-1}]_1$ be the kernel of $G(\mathbb{F}[\tau^{-1}]) \rightarrow G(\mathbb{F}), \: \tau^{-1} \mapsto 0$. Taking the orbit of the basepoint $\tau^0 \in \Gr$ under $G[\tau^{-1}]_1$ gives an open embedding into the neutral connected component of $\Gr$. We take the reduced ind-scheme structure on $G[\tau^{-1}]_1$ and identify it with a subscheme of $\Gr$, often called the big open cell.

The quotient $G \rightarrow G_{\ad}$ induces a map $\Gr \rightarrow \Gr_{\ad}$ to the affine Grassmannian for $G_{\ad}$. At the level of connected components, this is given by the natural map $\Omega \rightarrow \Omega_{\ad}$, which is neither surjective nor injective in general, but is surjective if $Z(G)$ is a torus. Each connected component of $\Gr$ maps via a universal homeomorphism onto its image $\Gr_{\ad}$, in a way which isomorphically identifies $G(\mathbb{F}[\![\tau]\!])$-orbits with $G_{\ad}(\mathbb{F}[\![\tau]\!])$-orbits.

\subsection{Bruhat--Tits groups}
Let $t$ be the standard coordinate function on $\mathbb{P}^1$, so that the complete local ring at $0$ is identified with $\mathbb{F}[\![t]\!]$ and the complete local ring at $\infty$ is identified with $\mathbb{F}[\![t^{-1}]\!]$. We identify $\mathbb{G}_m$ with $\mathbb{P}^1 \setminus \{0, \infty\}$.
Let
\begin{align*} I(0) &= \{g \in G(\mathbb{F}[\![t^{-1}]\!]) \: \mid \: g(0) \in B(\mathbb{F})\}, \quad \text{the Iwahori subgroup at } \infty, \\
I(1) &= \{g \in G(\mathbb{F}[\![t^{-1}]\!]) \: \mid \: g(0) \in U(\mathbb{F})\}, \quad \text{the unipotent radical of } I(0).
\end{align*}
We also have the analogous groups $I^{\opp}(0)$ and $I^{\opp}(1)$ constructed using the opposite groups $B_-$ and $U_-$ at $0$ instead of $\infty$. We identify all of these groups with their corresponding affine loop groups over $\mathbb{F}$.

Let $I(2) \subset I(1)$ be the subgroup as in \cite[\S 4.1.1]{XZ22} such that
$$I(1) / I(2) = \bigoplus_{\alpha \text{ affine simple}} U_\alpha \cong \bigoplus_{\alpha \text{ affine simple}} \mathbb{G}_a.$$ This is a slight departure from \cite[\S 1.2]{HNY} in the case when $G$ has nontrivial center, but when $G=\GL_n$ it is consistent with their own modification in \cite[\S 3]{HNY}. 
Here the trivializations of the $U_\alpha$ for $\alpha \in \Delta$ were chosen in Section \ref{sec:Crystal}. Letting $\theta$ denote the highest root, the unique affine simple root not belonging to $\Delta$ is trivialized by $u_{-\theta}(x t^{-1})$, $x \in \mathbb{G}_a$. By our choice of $I(2)$, the map $G \rightarrow G_{\ad}$ induces an isomorphism $I(1)/I(2) \cong I_{\ad}(1)/I_{\ad}(2)$.

Let $\mathcal{G} := \mathcal{G}(1,2)$ be the Bruhat--Tits group scheme over $\mathbb{P}^1$ as in \cite[\S 4.1.1]{XZ22} such that
$$\restr{\mathcal{G}}{\mathbb{G}_m} = G \times \mathbb{G}_m, \quad \mathcal{G}(\mathbb{F}[\![t^{-1}]\!]) = I(2), \quad \mathcal{G}(\mathbb{F}[\![t]\!]) = I^{\opp}(1).$$ Here $\mathcal{G}$ is constructed as a dilatation of $\mathcal{G}(1,1)$ along a subgroup of the fiber at $\infty$, where $\mathcal{G}(1,1)$ is the dilatation of $G \times \mathbb{P}^1$ along $U_- \times \{0\}$ and $U \times \{\infty\}$. By basic properties of dilatations, e.g.~\cite[\S 2.4, \S 2.5]{MRR23}, the map $G \times \mathbb{P}^1  \rightarrow G_{\ad} \times \mathbb{P}^1$ extends to a homomorphism $\mathcal{G} \rightarrow \mathcal{G}_{\ad}$.

\subsection{Moduli of bundles}
Let $\Bun$ be the moduli stack of $\mathcal{G}$-bundles on $\mathbb{P}^1$. We now review some of the geometry of $\Bun$; see \cite[\S 1.4]{HNY} and \cite[\S 7]{LamTemplier2024} for more details. As with $\Gr$, the connected components of $\Bun$ are indexed by $\Omega$.
For $\lambda \in \Omega$, let $\Bun^{\lambda}$ be the corresponding connected component. By \cite[Corollary 1.2]{HNY}, there is a canonical isomorphism $\Bun^0 \cong \Bun^\lambda$.
Let $\mathcal{E}_0 \in \Bun^0(\mathbb{F})$ correspond to the trivial bundle, and for $\lambda \in \Omega$, let $ \mathcal{E}_{\lambda} \in \Bun^{\lambda}(\mathbb{F})$ be the image of $\mathcal{E}_0$ under the isomorphism $\Bun^0 \cong \Bun^\lambda$.

The group $T = I^{\opp}(0)/I(1)^{\opp}$ acts on $\Bun$ by modifying $\mathcal{G}$-bundles at $0$, and the group $I(1)/I(2)$ acts on $\Bun$ by modifying $\mathcal{G}$-bundles at $\infty$. Taking the orbit of $\mathcal{E}_0$ under $T \times I(1)/I(2)$ gives an affine open immersion of the open cell $T \times I(1)/I(2)$ into $\Bun^0$.

\begin{assumption} \label{Assump1}
    For the rest of the paper unless otherwise stated, we assume that $Z(G)$ is a torus and that we have chosen an isomorphism $T \cong Z(G) \times T_{\ad}$ as in \eqref{eq-split}. We also fix a minuscule parabolic $P \subset G$, associated to a simple root $\alpha_i \in \Delta = \Delta_{\ad}$ with minuscule fundamental coweight $\omega_i^\vee \in X_*(T_{\ad})$. Additionally, by \cite[Lemma 6.30(c)]{LamTemplier2024} we have $w_P^{-1} \alpha_i = -\theta$, the negative of the highest root,  and we choose the sign in $u_{-\theta} \colon \mathbb{G}_a \rightarrow U_{-\theta}$ so that $\dot{w}_P^{-1}u_{\alpha_i}(x) \dot{w}_P = u_{-\theta}(-x)$ for all $x \in \mathbb{G}_a$, in accordance with \cite[Eqn.~(7.8.2)]{LamTemplier2024}.
\end{assumption}

From the isomorphism $T \cong Z(G) \times T_{\ad}$ we get a natural lift of $\omega_i^\vee$ to $\tilde{\omega}_i^\vee \in X_*(T)$. By Corollary \ref{cor:splitting-of-center-of-levi-part} we have an isomorphism $Z(L) \cong Z(G) \times Z(L_{\ad})$, and we note that the resulting inclusion $Z(L_{\ad}) \rightarrow T$ agrees with the composition $\tilde{\omega}_i^\vee \circ \restr{\alpha_i}{Z(L_{\ad})}$.
Let $$\dot{\kappa} := \tau^{-\tilde{\omega}_i^\vee} \dot{w}_P = \tilde{\omega}_i^\vee(\tau) \dot{w}_P \in G(\mathbb{F}[\tau^{\pm 1}]).$$
If $\kappa \in \Omega$ is the element represented by $\tilde{\omega}_i^\vee$, then $\mathcal{E}_\kappa \in \Bun^{\kappa}(\mathbb{F})$ is obtained from gluing the trivial bundle on the formal disc $\mathrm{Spec} \: \mathbb{F}[\![t^{-1}]\!]$ at $\infty$ along the transition function $\dot{\kappa}$ over the punctured disc $\mathrm{Spec} \: \mathbb{F}(\!(t^{-1})\!)$, where we identify the local parameters $\tau^{-1}$ and $t^{-1}$. 

\subsection{Hecke stacks} In this subsection we follow \cite[\S 5]{HNY} and \cite[\S 7]{LamTemplier2024}, with a small change to the latter.
Consider the Hecke stack $\Hk$ which sends an affine $\mathbb{F}$-scheme $S$ to the groupoid
$$\Hk(S) := \left\{ (\mathcal{E}, a, \varphi) \: \mid \:  \mathcal{E} \in \Bun(S), \: a \colon S \rightarrow \mathbb{G}_m, \: \varphi \colon \restr{\mathcal{E}}{(\mathbb{P}^1 \setminus \{a\}) \times S} \cong \restr{\mathcal{E}_\kappa}{(\mathbb{P}^1 \setminus \{a\}) \times S} \right\}.$$
There are projection maps $$\pr_1 \colon \Hk \rightarrow \Bun, \: \pr_1(\mathcal{E}, a, \varphi) = \mathcal{E}, \quad \pr_2 \colon \Hk \rightarrow \mathbb{G}_m, \: \pr_2(\mathcal{E}, a, \varphi) = a.$$

Let $\Hk^\circ$ be the inverse image of the open cell $T \times I(1)/I(2) \subset \Bun^0$ under $\pr_1$. Denote by $U/[U,U]$ the quotient of $U$ by the subgroup generated by the non-simple positive roots. Then there is a canonical isomorphism $I(1)/I(2) \cong U/[U,U] \times U_{-\theta}$. The restriction of $\pr_1$ to $\Hk^\circ$ then factors as a product
$$(f_T, f_+, f_0) \colon  \Hk^\circ \rightarrow T \times U/[U,U] \times U_{-\theta}.$$

For any point $a \in \mathbb{G}_m$ (over any affine base $S$), we use the local parameter $\tau = 1 - t/a$ at $a$. Then as shown in \cite[Eqn.~(7.3.4)]{LamTemplier2024}, $\dot{\kappa}$ can be viewed as an isomorphism  $\restr{\mathcal{E}_0}{\mathbb{P}^1 \setminus \{a\}} \cong \restr{\mathcal{E}_\kappa}{\mathbb{P}^1 \setminus \{a\}}$. This isomorphism is given by $\dot{\kappa} = \tilde{\omega}_i^\vee(\tau) \dot{w}_P$ in the local chart on $\mathbb{A}^1 \setminus \{a\}$, and in the local chart on the disc at $\infty$ by %starts in $\tau$ coordinates, but must end in $t$ coordinates,
\begin{equation} \label{inf-iso} (\tilde{\omega}_i^\vee(t) \dot{w}_P)^{-1}(\tilde{\omega}_i^\vee(\tau) \dot{w}_P) = \dot{w}_P^{-1} \tilde{\omega}_i^\vee(\tau/t) \dot{w}_P.\end{equation}
Using that $\tau^{-1} = -at^{-1} +O((t^{-1})^2)$ as in \cite[Eqn.~(7.3.3)]{LamTemplier2024}, it follows that the value of \eqref{inf-iso} at $t=\infty$ is given by
\begin{equation} \label{ca} c_a :=  \dot{w}_P^{-1} \tilde{\omega}_i^\vee(-a^{-1}) \dot{w}_P \in T.\end{equation}

Any point in $\Hk^\circ$ lying over $a \in \mathbb{G}_m$ can be obtained by precomposing $\dot{\kappa}$ with some element $g(\tau^{-1}) \in \mathrm{Aut}(\restr{\mathcal{E}_0}{\mathbb{P}^1 \setminus \{a\}}) =G[\tau^{-1}]$.
The level structure associated to $\dot{\kappa} g(\tau^{-1})$ is defined by evaluation of $(\dot{\kappa} g(\tau^{-1}))^{-1}$ at $t=0$ and at $t=\infty$, but contrary to what is claimed in \cite[\S 7]{LamTemplier2024}, the isomorphism \cite[Eqn.~(7.3.4)]{LamTemplier2024} does change the level structure at $\infty$, by $c_a^{-1}$. Therefore, the resulting level structure is not $(g(1)^{-1} \dot{w}_P^{-1}, g(0)^{-1})$, but rather
\begin{equation} \label{level-str} (g(1)^{-1} \dot{w}_P^{-1}, g(0)^{-1}c_a^{-1}) \in B_- \times U. \end{equation}

As in \cite[\S 7]{LamTemplier2024}, one rigidifies the moduli problem defining $\Hk^\circ$ by requiring that the modification $\dot{\kappa}^{-1} \circ \varphi$ of $\mathcal{E}_0$ is trivial at $\infty$, thereby obtaining an element of $G[\tau^{-1}]_1$. Let $h(\tau^{-1}) = c_a g(\tau^{-1}) g(0)^{-1} c_a^{-1} \in G[\tau^{-1}]_1$. Letting $a$ vary and taking into account the level structures at $0$ and $\infty$ in \eqref{level-str} leads to the following.

\begin{proposition} \label{Hkcirc}
    There is an isomorphism
    $$\Hk^\circ \cong \{ h(\tau^{-1})  \in G[\tau^{-1}]_1 \: \mid \: h(1) \in \dot{w}_P^{-1} B_- U\} \times \mathbb{G}_m$$ characterized by the property that for $a \in \mathbb{G}_m$, the image of $h(\tau^{-1})$ in the fiber of $\pr_2$ over $a$ is $\dot{\kappa} h(\tau^{-1}) \colon \restr{\mathcal{E}_0}{\mathbb{P}^1 \setminus \{a\}} \cong \restr{\mathcal{E}_\kappa}{\mathbb{P}^1 \setminus \{a\}}$.
\end{proposition}

\begin{proof}
    This is proved in \cite[Eqn.~(7.3.5)]{LamTemplier2024} over $\mathbb{C}$ and for $G$ of adjoint type, but the same proof applies over $\mathbb{F}$ and for an arbitrary split reductive group. In fact, loc.~cit. is a straightforward modification of the argument in \cite[\S 5.2]{HNY}, which works over $\mathbb{F}$ but uses the neutral component $\Bun^0$ instead of $\Bun^\kappa$. Here the missing factor of $c_a^{-1}$ in the level structure at $\infty$ makes no difference, since $c_aB_-Uc_a^{-1} = B_-U$.
\end{proof}

\begin{remark} \label{LTrem}
The factor $c_a^{-1}$ does make a difference in the formula for $f_T \colon \Hk^\circ \rightarrow T$ in \cite[Lemma 7.4]{LamTemplier2024}, as follows. For $h(\tau^{-1})$ representing a point of $\Hk^\circ$ over $a \in \mathbb{G}_m$ as in Proposition \ref{Hkcirc}, write $$h(1) = c_ag(1)g(0)^{-1}c_a^{-1} = \dot{w}_P^{-1} b_- u$$ for $b_- \in B_-$ and $u \in U$. Then $\dot{w}_P^{-1}b_- = c_a g(1)$ and $u = g(0)^{-1}c_a^{-1}$. Using the definition of $c_a$, the former equation gives $g(1)^{-1} \dot{w}_P^{-1} = (b_-)^{-1} \tilde{\omega}_i^\vee(-a^{-1})$. From the definition of the level structures in \eqref{level-str}, it follows that
\begin{equation} f_T(h,a) = \tilde{\omega}_i^\vee(-a^{-1}) (b_-)^{-1} \textrm{ mod } U_- \in B_-/U_- \cong T.
\end{equation}
The formula $f_+(h,a) = u \textrm{ mod } [U,U]$  in \cite[Lemma 7.4]{LamTemplier2024} is unchanged, as is the formula for $f_0(h,a)$ (as the latter correctly takes into account the difference \cite[Eqn.~(7.3.3)]{LamTemplier2024} between $\tau^{-1}$ and $t^{-1}$ at $\infty$).
\end{remark}

\begin{remark} \label{NormRemark}
The omission of the factor $c_a^{-1}$ 
in \cite{LamTemplier2024} results in an embedding of $T$ into the open cell of $\Bun^\kappa$ which is sheared by a $\mathbb{G}_m$-dependent element. 
This is responsible for the appearance of the multiplicative $D$-module on $\mathbb{G}_m$ as an extra factor in the proof of the weighted theorem \cite[Theorem 10.5]{LamTemplier2024}. The \'{e}tale avatar of the multiplicative $D$-module is a Kummer sheaf on $\mathbb{G}_m$. One does not expect an extra Kummer factor outside the Kloosterman sheaf to appear in the final determination of the weighted geometric crystal, and we will see in Theorem \ref{KLtoWC} below that our correction to $f_T$ indeed resolves this issue.    

While our taking into account the factor $c_a^{-1}$ ultimately removes the relative $\mathbb{G}_m$-dependence of the embedding of $T$, there are still nontrivial translations of $T$ defined over $\mathbb{Z}$ (for example, multiplication by $\tilde{\omega}_i^\vee(-1)$). We do not claim that \cite{LamTemplier2024} and \cite{HNY} use identical embeddings of the open cell, so our Kloosterman sheaves in Definition \ref{KDef} below should more accurately be considered ``Lam--Templier'' normalized, and having the level structure used by Xu--Zhu \cite{XZ22} (which is also consistent with Lam--Templier). This may cause minor differences for some groups between the Kloosterman sheaves of \cite{HNY} and the \'{e}tale version of \cite{LamTemplier2024}, e.g.~differing by pullback along multiplication by $-1$ on $\mathbb{G}_m$ or by an order-two rank-one local system on $\mathrm{Spec}(\mathbb{F})$, c.f.~also Remark \ref{LSpecF}. We do not pursue a precise comparison here.
\end{remark}

\subsection{Kloosterman sheaves}
For any point $a \in \mathbb{G}_m$, the affine Grassmannian $\Gr$ can be identified with the functor which sends an $\mathbb{F}$-scheme $S$ to the set of isomorphism classes
$$\left\{(\mathcal{E}, \varphi) \: \mid \: \mathcal{E } \text{ a } G\text{-bundle on }\mathbb{P}^1 \times S, \: \varphi \colon \restr{\mathcal{E}}{(\mathbb{P}^1 \setminus \{a\}) \times S} \cong \restr{\mathcal{E}_0}{(\mathbb{P}^1 \setminus \{a\}) \times S}\right\}.$$
By composing the modification $\varphi$ above with $\dot{\kappa}$, the fiber of $\pr_2 \colon \Hk \rightarrow \mathbb{G}_m$ over $a$ is identified with $\Gr$.
These fibers can be trivialized globally as in Proposition \ref{Hkcirc}, leading to identifications \begin{equation} \label{Hkincllusion} \Hk^\circ \overset{\ref{Hkcirc}}{\subset} G[\tau^{-1}]_1 \times \mathbb{G}_m \xrightarrow{(\cdot \dot{\kappa}, \: \mathrm{id})} \Gr^{\kappa} \times \mathbb{G}_m \subset \Gr \times \mathbb{G}_m= \text{Hk}.\end{equation}
Here  $\Gr^{\kappa}$ is the connected component of $\Gr$ containing the point $\tilde{\omega}_i^\vee(\tau)$.

For $\mu \in X_*(T)^+$ such that $\Gr_\mu \subset \Gr^{\kappa}$, let $\overline{\Gr}_{\mu} \subset \Gr_G$ be the closure of $\Gr_{\mu}$, and let $j^\mu \colon \Gr_{\mu} \rightarrow \overline{\Gr}_{\mu}$ be the inclusion. Using the intermediate extension functor, we define the perverse sheaf 
\begin{equation} \label{eq--normalize}
    \IC_\mu := j_{!*}^\mu (\Ql(\tfrac{1}{2}\dim \Gr_\mu)[\dim \Gr_\mu]).
\end{equation} The Tate twist $\Ql(\tfrac{1}{2}\dim \Gr_\mu)$ ensures that $\IC_\mu$ is pure of weight zero, in agreement with the convention in \cite[Remark 2.10]{HNY}.

Using \eqref{Hkincllusion}, we may define $$\Hk^\circ_\mu = \Hk^\circ \cap (\overline{\Gr}_\mu \times \mathbb{G}_m).$$
Taking the restrictions of maps defined earlier, we have the following diagram.
\begin{equation} \label{Hkdiagram} \xymatrix{
& \Hk^\circ_\mu\ar[ld]_-{(f_T, f_+, f_0)} \ar[rd]^-{\pr_2} & \\
 T \times U/[U,U] \times U_{-\theta} & & \mathbb{G}_m
}\end{equation}

\begin{definition} \label{KDef} Let $\text{AS}_{-\theta} = (u_{-\theta})_*(\text{AS}_{\psi})$ where $u_{-\theta} \colon \mathbb{G}_a \rightarrow U_{-\theta}$ is the trivialization chosen in Assumption \ref{Assump1}. Let $\text{AS}_{\phi} = \phi^*(\text{AS}_{\psi})$, where we view $\phi$ as a homomorphism $U/[U,U] \rightarrow \mathbb{G}_a$. The Kloosterman sheaf associated to any $\mu \in X_*(T)^+$ such that $\Gr_\mu \subset \Gr^{\kappa}$ is
   $$\text{Kl}_{\hat{G}}^{\mu}(\phi, \chi)
    :=  R\pr_{2, !} (f_T^* \tilde{\omega}_i^{\vee}(-1)^* \K_\chi \otimes f_+^* \text{AS}_{\phi} \otimes f_0^* \text{AS}_{- \theta} \otimes \restr{(\IC_\mu \boxtimes \Ql)}{\Hk^\circ}).$$
    Here $\IC_\mu \boxtimes \Ql$ is a sheaf on $\Hk^\circ$ via \eqref{Hkincllusion}, which we normalize as in \eqref{eq--normalize}.
\end{definition}

\begin{remark} \label{LSpecF}
The sheaf $\tilde{\omega}_i^{\vee}(-1)^* \K_\chi$ is the pullback of $\K_\chi$ along multiplication by the order-two element $\tilde{\omega}_i^{\vee}(-1) \in T$. Let $\mathcal{L}_\chi$ be the local system on $\mathrm{Spec}(\mathbb{F})$ which sends a geometric Frobenius element to $\chi(\tilde{\omega}_i^{\vee}(-1)) \in \{\pm 1\}$ (so $\mathcal{L}_\chi$ is trivial if $\chi(\tilde{\omega}_i^{\vee}(-1)) = 1$). By abuse of notation we also use $\mathcal{L}_\chi$ to denote its pullback to any $\mathbb{F}$-scheme. The multiplicative property of $ \K_\chi$ implies that $\tilde{\omega}_i^{\vee}(-1)^* \K_\chi \cong K_\chi \otimes \mathcal{L}_\chi$. Hence by the projection formula, we can replace $f_T^* \tilde{\omega}_i^{\vee}(-1)^* \K_\chi$ with $f_T^* \K_\chi$ in the definition of $\text{Kl}_{\hat{G}}^{\mu}(\phi, \chi)$ at the cost of taking the tensor product of $\text{Kl}_{\hat{G}}^{\mu}(\phi, \chi)$ with $\mathcal{L}_\chi$. 

The appearance of the factor $\mathcal{L}_\chi$ is not surprising; a similar multiplication by $(-1)$-map appears in the explicit calculation of the Kloosterman sheaf for the standard representation of $\GL_n$ in \cite[\S 3.2]{HNY}. In any case, while $\text{Kl}_{\hat{G}}^{\mu}(\phi, \chi)$ a priori belongs only to $D^b_{c}(\mathbb{G}_m, \Ql)$, by \cite[Theorem 1, \S 5.2]{HNY} it is in fact a local system, pure of weight zero and concentrated in cohomological degree zero.
\end{remark}

\section{Proof of the main result} \label{Sec:proof} 
It follows from Proposition \ref{Hkcirc} and \eqref{Hkincllusion} that for any $\mu \in X_*(T)^+$ such that $\Gr_\mu \subset \Gr^{\kappa}$, the natural map $\Hk^\circ_\mu \rightarrow \Hk^\circ_{\mu_{\ad}}$ is a universal homeomorphism, so as far as \'etale cohomology is concerned we may treat it like an isomorphism. In fact, if $\mu$ is minuscule then $\overline{\Gr}_{\mu} = \overline{\Gr}_{\mu_{\text{ad}}} \cong G/P_{\mu}$, so $\Hk^\circ_\mu \rightarrow \Hk^\circ_{\mu_{\ad}}$ is an isomorphism. However, $\Hk^\circ_\mu$ comes equipped with a map to $T$ instead of $T_{\ad}$, which we will analyze below in the minuscule case.

By \cite[Lemma 6.2]{LamTemplier2024}, any $x \in X_{\ad}$ may be written uniquely as $x= u_1 t \dot{w}_P u_2$ where $t \in Z(L_{\ad})$, $u_1 \in U_{P_{\ad}}$ (the unipotent radical of $P_{\ad}$) and $u_2 \in U_{\ad} \cap (B_{\text{ad}})_- \cdot \dot{w}_P^{-1} \cdot (B_{\text{ad}})_-$. Note that the restriction $\alpha_i \colon Z(L_{\ad}) \rightarrow \mathbb{G}_m$ is an isomorphism. By \cite[Lemma 7.8]{LamTemplier2024}, there is an isomorphism
\begin{align} \label{XadIso} \begin{split} \tilde{\iota} \colon X_{\ad} & \xrightarrow{\sim} \Hk^\circ_{\omega_i^\vee} \subset G_{\ad}[\tau^{-1}]_1 \times \mathbb{G}_m \\ x & \mapsto (\dot{\kappa}_{\ad}^{-1} t^{-1} u_1 t \dot{\kappa}_{\ad}, \alpha_i(t)). \end{split} \end{align}
Here $\dot{\kappa}_{\ad} = {\omega}_i^\vee(\tau) \dot{w}_P \in G_{\ad}(\mathbb{F}[\tau^{\pm 1}])$ is the image of $\dot{\kappa}$ in the adjoint quotient. 

\begin{proposition} \label{XadMaps}
    Identify $\Hk^\circ_{\tilde{\omega}_i^\vee} \cong \Hk^\circ_{\omega_i^\vee}$ as above (noting that $(\tilde{\omega}_i^\vee)_{\ad} = \omega_i^\vee$), so that the maps $(f_T, f_+, f_0)$ are defined on $\Hk^\circ_{{\omega}_i^\vee}$. Then with respect to the isomorphism \eqref{XadIso}, we have
    \begin{align*}
        f_T(\tilde{\iota}(x)) &= \tilde{\omega}_i^\vee(-1)\gamma_{\ad}^T(x)^{-1} \\
        f_+(\tilde{\iota}(x)) &= u_2^{-1} \textnormal{ mod } [U,U] \\
        f_0(\tilde{\iota}(x)) &= u_{-\theta}(-\phi(u_1)).
    \end{align*}
Here $\gamma_{\ad}^T \colon X_{\ad} \rightarrow T$ is the lifted weight map from Corollary \ref{cor-weightSplit}, and the element $t \in Z(L_{\ad})$ is viewed as an element of $T$ via the splitting \eqref{eq-split}.
\end{proposition}

\begin{proof}
    The formulas for $f_+$ and $f_0$ are proved in \cite[Proposition 7.9]{LamTemplier2024}.
    There it is also shown that, with respect to their parametrization of $\Hk^
    \circ$, one has
    $$f_{T_{\ad}}(\tilde{\iota}(x)) = t x^{-1} \text{ mod } (U_\text{ad})_- = t \gamma_{\ad}(x)^{-1} \in T_{\ad}.$$
    This is obtained by noting that $(t^{-1}x)(u_2)^{-1} \in B_- U$ and then applying the formula for $f_T$ in \cite[Lemma 7.4]{LamTemplier2024} (which simply records the inverse $(tx^{-1})^{-1} =(b_-)^{-1}$ of the factor in $B_- \textrm{ mod } U_-$). As we noted in Remark \ref{LTrem}, the formula for $f_T$ should actually be $\omega_i^{\vee}(-\alpha_i(t)^{-1})(b_-)^{-1} \textrm{ mod } U_-$. Taking this into account, we get the desired formula $f_{T_{\ad}}(\tilde{\iota}(x)) = \omega_i^\vee(-1) \gamma^{T{_{\ad}}}(x)^{-1}$ in the adjoint case.

    In the general case, for $x = u_1t\dot{w}_Pu_2 \in X_\ad$, the choice of \eqref{eq-split} determines an isomorphism $X = Z(G) \times X_{\ad}$, and hence a lift $\tilde{x}$ of $x$ to $X$. Explicitly, $\tilde{x} = u_1 \tilde{t} \dot{w}_P u_2$ where $\tilde{t} = \tilde{\omega}_i^{\vee}(\alpha_i(t))$. An inverse to $\Hk^\circ_{\tilde{\omega}_i^\vee} \cong \Hk^\circ_{\omega_i^\vee}$ with respect to Proposition \ref{Hkcirc} is then given by sending $\tilde{\iota}(x)$ to $(\dot{\kappa}^{-1} \tilde{t}^{-1} u_1 \tilde{t} \dot{\kappa}, \alpha_i(t))$. 
    Now by the same argument as in \cite[Proposition 7.9]{LamTemplier2024}, and taking into account our Remark \ref{LTrem}, 
    we have $f_T(\tilde{x}) = \tilde{\omega}_i^\vee(-1) (\tilde{x})^{-1} \mod U_-$. The map $\gamma_{\ad}^T$ was constructed in Corollary \ref{cor-weightSplit} by using the decomposition $X_{\ad} = Z(L_\ad) \times (B_{\ad})_-^{w_P}$, and the isomorphism $(B_{\ad})_-^{w_P} \cong B_-^{w_P}$ in Lemma \ref{Bad}. Writing $\tilde{x} = \tilde{t} \cdot ((\tilde{t})^{-1}u_1 \tilde{t}) \dot{w}_P u_2$, the 
    second factor lies in $B_-^{w_P}$ and it follows that $f_T(\tilde{x}) = \tilde{\omega}_i^\vee(-1)\gamma_{\ad}^T(x)^{-1}$.
\end{proof}

\begin{theorem} \label{KLtoWC} 
View $\alpha_i$ as an isomorphism $Z(L_{\ad}) \rightarrow \mathbb{G}_m$ by restriction from $T_{\ad}$, and let $\K_{\chi}^{Z(G)}$ be the restriction of $\K_{\chi}$ from $T$ to $Z(G)$. Then with respect to our chosen decomposition $Z(L) \cong Z(G) \times Z(L_{\ad})$, there exists an isomorphism
    $$\WC_G^P(\phi, \chi)  \cong \K_{\chi}^{Z(G)} \boxtimes \alpha_i^*\textnormal{Kl}_{\hat{G}}^{\tilde{\omega}_i^\vee}(\phi, \chi^{-1}).$$
\end{theorem}

\begin{proof}
By Proposition \ref{XadMaps}, the diagram
$$
\xymatrix{
& X_{\ad} \ar[ld]_-{f_\ad} \ar[d]^-{\pi_{\ad}} \ar[rd]^-{\gamma^T_{\ad}} & \\
\mathbb{A}^1 & Z(L_{\ad}) & T
}
$$
may be identified with the diagram
$$
\xymatrix{
& \Hk^\circ_{\tilde{\omega}_i^\vee} \ar[ld]_-{-(\phi(f_+)+(u_{-\theta})^{-1} \circ f_0)} \ar[d]^-{\pr_2} \ar[rd]^-{\frac{1}{f_T} \cdot \tilde{\omega}_i^\vee(-1)} & \\
\mathbb{G}_a & \mathbb{G}_m & T
}
$$
In the above diagram, $(u_{-\theta})^{-1} \colon U_{-\theta} \rightarrow \mathbb{G}_a$ is the inverse of the trivialization $u_{-\theta}$, and $\frac{1}{f_T}$ is the composition of $f_T \colon \Hk^\circ_{\tilde{\omega}_i^\vee}  \rightarrow T$ with the inversion map on $T$. Both of the outer maps are constructed using two maps to the target and the group structure on the target.

Since $\text{AS}_{\psi}$ is a character sheaf, its pullback along the addition map $\text{add} \colon \mathbb{G}_a \times \mathbb{G}_a \rightarrow \mathbb{G}_a$ satisfies $\text{add}^{-1}(\text{AS}_{\psi}) \cong \text{AS}_{\psi} \boxtimes \text{AS}_{\psi}$. The sheaf $f_+^* \text{AS}_{\phi} \otimes f_0^* \text{AS}_{- \theta}$ appearing in Definition \ref{KDef} is thus canonically isomorphic to $(\phi(f_+)+(u_{-\theta})^{-1} \circ f_0)^{-1} (\text{AS}_{\psi})$. Additionally, since $\tilde{\omega}_i^\vee$ is minuscule,  $\overline{\Gr}_{\tilde{\omega}_i^\vee} \cong G/P$ is smooth, so the object $ \restr{(\IC_{\tilde{\omega}_i^\vee} \boxtimes \Ql)}{\Hk^\circ}$ in Definition \ref{KDef} is $\Ql(\tfrac{d}{2})[d]$ supported on $\Hk^\circ_{\tilde{\omega}_i^\vee}$.
By similar reasoning, we have 
$$(\frac{1}{f_T} \cdot \tilde{\omega}_i^\vee(-1))^{*}(\K_\chi) \cong f_T^{*} \tilde{\omega}_i^\vee(-1)^*(\K_{\chi^{-1}}).$$
It follows that
$$R\pi_{\ad, !} ((\gamma_{\ad}^T)^*\K_{\chi} \otimes f_{\ad}^* \A) (\tfrac{d}{2})[d] \cong \text{Kl}_{\hat{G}}^{\tilde{\omega}_i^\vee}(\phi, \chi^{-1}).$$
We may thus conclude using  Lemma \ref{lem-WCad}.
\end{proof}

As a consequence, $\WC_G^P(\phi, \chi)$ is a local system on $Z(L)$ as opposed to a general object of $D^b_c(Z(L), \Ql)$. Applying the function-sheaf dictionary, we also get the following result.

\begin{corollary} \label{maincor}
    For any $z \in Z(L)(\mathbb{F})$, write $z = z_0 z_1$ where $z_0 \in Z(G)(\mathbb{F})$ and $z_1 \in Z(L_{\ad})(\mathbb{F})$. Then
$$ \textnormal{Kl}_{\hat{G}}^{\tilde{\omega}_i^\vee}(\phi, \chi^{-1}; \alpha_i(z))  = \overline{\chi(z_0)} (-1)^{\dim G/P} q^{\tfrac{\dim G/P}{2}} \besselFunction_{\tau, \fieldCharacter}\left(z \specialWeylElement{P}\right).
$$

\end{corollary}

\begin{proof}
    This follows from Theorem \ref{KLtoWC} and Lemma \ref{Trace-lem}.
\end{proof}

\begin{remark} \label{notTorus} We now explain how to generalize the main result to the case where $Z(G)$ is (possibly) not a torus.
    First, choose an embedding $(G,T) \rightarrow (G', T')$ as in Lemma \ref{lemm-G'}. Choose a splitting of the inclusion $T \rightarrow T'$ and let $\chi' \colon T'(\mathbb{F}) \rightarrow \Ql^{\times}$ be the corresponding extension of $\chi$ as in Lemma \ref{lem-WC'}. Then choose an isomorphism $T' \cong Z(G') \times T'_{\ad}$ as in Lemma \ref{lem-split} and let $\tilde{\omega}_i^{\vee} \in X_*(T')$ be the resulting lift of $\omega_i^{\vee}$. By Lemma \ref{lem-WC'} and Theorem \ref{KLtoWC} we have
    $$\WC_G^P(\phi, \chi) \cong  \restr{\K_{\chi'}^{Z(G')} \boxtimes \alpha_i^*\textnormal{Kl}_{\hat{G}'}^{\tilde{\omega}_i^\vee}(\phi, (\chi')^{-1})}{Z(L)}.$$
    Here we view $\alpha_i$ as an isomorphism $Z((L')_{\ad}) \rightarrow \mathbb{G}_m$. Consequently, if we replace $\hat{G}$ with $\hat{G}'$ in the Kloosterman sheaf then Corollary \ref{maincor} still holds (and the right side is unchanged).
\end{remark}

\begin{remark} \label{rem:Norm}
	By normalizing Kloosterman sheaves to have weight zero, the eigenvalues of Frobenius are unitary. Up to a factor of complex norm 1, the relation in Corollary \ref{maincor} is ``trace of Frobenius equals square root of double coset size times value of Bessel function.'' We explain why the Bessel function side of this equality is natural. In what follows, we identify algebraic groups with their $\mathbb{F}$-points.
	
	Consider the space of all $\fieldCharacter \circ \phi$-Whittaker functions $\Ind{U}{G}{\fieldCharacter \circ \phi}$. This space is equipped with a natural inner product which is invariant under the $G$-action:
	$$\innerproduct{W_1}{W_2} = \sum_{g \in U \backslash G} W_1\left(g\right) \conjugate{W_2\left(g\right)}.$$
	The space $\Ind{U}{G}{\fieldCharacter \circ \phi}$ is multiplicity free, see e.g.~\cite[Theorem 12.3.4]{DJ20}. For any $\left(U, \fieldCharacter \circ \phi\right)$-generic representation $\tau$ let $\Whittaker\left(\tau, \fieldCharacter \circ \phi\right)$ be the unique subspace of $\Ind{U}{G}{\fieldCharacter \circ \phi}$ that is isomorphic to $\tau$. This is the \emph{Whittaker model of $\tau$} with respect to the character $\fieldCharacter \circ \phi$. By Frobenius reciprocity, for every such $\tau$ there exists a unique (up to scalar) element in $\Whittaker\left(\tau, \fieldCharacter \circ \phi\right)$ that is right $(U, \fieldCharacter \circ \phi)$-equivariant.
	
	Consider the subspace $\mathcal{BW}_{U}^G\left(\psi \circ \phi\right) \subset \Ind{U}{G}{\fieldCharacter \circ \phi}$ consisting of functions that are right $\left(U, \fieldCharacter \circ \phi\right)$-equivariant. We have two natural bases for this space:
	\begin{enumerate}
		\item We say that a double coset $U g U \in U \backslash G \slash U$ is \emph{compatible with $\fieldCharacter \circ \phi$} if there exists a nonzero bi-$\left(U, \fieldCharacter\circ\phi\right)$-equivariant function on $G$ that does not vanish on $U g U$. For each such compatible coset we choose a representative of the form $tw$ where $t \in T$ and $w$ is a lift of an element of the Weyl group and we let $W_{U tw U}$ be the unique $\left(U, \fieldCharacter \circ \phi\right)$ bi-equivariant function on $G$ supported on $U tw U$ whose value at $tw$ is $1$. Then $$\left(W_{U tw U} \mid U tw U \text{ is } \fieldCharacter \circ \phi \text{ compatible}\right)$$ is a basis for $\mathcal{BW}_{U}^G\left(\psi \circ \phi\right)$.
		\item For every irreducible $\left(U, \fieldCharacter\circ\phi\right)$-generic representation $\tau$ there exists a unique element $\besselFunction_{\tau, \fieldCharacter} \in \Whittaker\left(\tau, \fieldCharacter \circ \phi\right)$ that is $\left(U, \fieldCharacter\circ\phi\right)$-equivariant from the right such that its value at the identity is $1$. Then $$\left(\besselFunction_{\tau, \fieldCharacter} \mid \tau \text{ is } \fieldCharacter \circ \phi \text{ generic}\right)$$
		is also a basis for $\mathcal{BW}_{U}^G\left(\psi \circ \phi\right)$.
	\end{enumerate}
	Both bases are orthogonal bases with respect to $\innerproduct{\cdot}{\cdot}$. Indeed, it is clear that $\left(W_{U tw U}\right)$ is orthogonal, and $\left(\besselFunction_{\tau, \fieldCharacter}\right)$ is also orthogonal because matrix coefficients of non-isomorphic representations are orthogonal. However, the bases are not orthonormal: we have
	$$\innerproduct{W_{U tw U}}{W_{U tw U}}  = \frac{\sizeof{U tw U}}{\sizeof{U}} = \frac{\sizeof{U}}{\sizeof{\left(w^{-1} U w\right) \cap U}},$$ and by writing $\besselFunction_{\tau, \fieldCharacter}$  as a matrix coefficient and using the Schur orthogonality relations it can be shown that  $$\innerproduct{\besselFunction_{\tau, \fieldCharacter}}{\besselFunction_{\tau, \fieldCharacter}} = \frac{\grpIndex{G}{U}}{\dim \tau}.$$
	Thus, if we define $$W'_{U tw U} = \left(\frac{\sizeof{\left(w^{-1} U w\right) \cap U}}{\sizeof{U}}\right)^{1/2} \cdot W_{U tw U}$$ then the basis $\left(W'_{U tw U}\right)$ is an orthonormal basis.

    It follows from the definitions that for $\tau$ a generic principal series representation and setting $t=z$, $w=\dot{w}_P$ as in Corollary \ref{maincor}, we have
	$$\innerproduct{\besselFunction_{\tau, \fieldCharacter}}{W'_{U z\dot{w}_P U}} = q^{\tfrac{\dim G/P}{2}} \besselFunction_{\tau, \fieldCharacter}\left(z \specialWeylElement{P}\right).$$ Thus, Corollary \ref{maincor} implies that the above inner product is, up to an explicit factor of complex norm $1$, equal to the trace of Frobenius on a Kloosterman sheaf.
    
	\end{remark}

\section{Quasi-split groups} \label{Sec:Qsplit}
We continue using the same notation surrounding $G$ and its auxiliary data from Section \ref{sec:K}. In particular, the derived subgroup is almost simple and $Z(G)$ is a torus. 

Suppose we have a nontrivial automorphism of the Dynkin diagram of $G$ which fixes $\alpha_i$, where $\omega_i^{\vee}$ is minuscule. Let $\sigma$ be a pinning-preserving automorphism of $G$ which realizes this automorphism of the Dynkin diagram. Since we assume that our trivializations of root groups are realized by homomorphisms from $\SL_2$ then $\sigma$ necessarily preserves the trivializations of the negative simple root groups as well. 

We refer to \cite[Figure 2]{LamTemplier2024} for all possible Dynkin diagrams with minuscule nodes, but note that this table gives the minuscule fundamental weights (not coweights). It follows from inspection that $\sigma$ has order $2$ and there are three possibilities:

\begin{enumerate}
    \item $G$ is of type $A_n$ where $n \geq 3$ is odd, and $\sigma$ realizes the reflection about the center.
    \item \label{case3} $G$ is of type $D_n$ where $n \geq 4$, and $\sigma$ realizes the reflection which swaps the two nodes at the end of the branch.
    \item \label{case2} $G$ is of type $D_4$ and $\sigma$ realizes either of the two reflections not covered in case (\ref{case3}).
\end{enumerate}

It is well-known (and crucial for us) that in the three cases above, $\sigma$ also preserves the trivializations of the highest and lowest roots $\pm \theta$. (This can be deduced by writing a highest root vector in terms of Lie brackets of simple root vectors and applying $\sigma$, see e.g.~\cite[\S 8.3]{Kac90} for similar arguments.)

Recall that $\mathbb{F}$ is a finite field of order $q$, and let  $\mathbb{E}/\mathbb{F}$ be a quadratic field extension.
We define the quasi-split group $\mathbb{G}$ over $\mathbb{F}$ by taking $(\sigma \circ \mathrm{Frob}_q)$-invariants in the Weil restriction of $G$ along $\mathbb{E} / \mathbb{F}$. Equivalently, for any $\mathbb{F}$-algebra $R$ we have the $\mathbb{F}$-algebra automorphism $\overline{(-)}$ of $\mathbb{E} \otimes_{\mathbb{F}} R$ determined by $\overline{(a \otimes r)} = a^q \otimes r$, and we set
\begin{equation} \label{Geq} \mathbb{G}(R) = \{ g \in G(\mathbb{E} \otimes_{\mathbb{F}} R) \: \mid \: \sigma(\overline{g}) = g \}.\end{equation}
Since $B$, $U$, and $T$ are preserved by $\sigma$, the same equations allow us to define a Borel $\mathbb{B} \subset \mathbb{G}$, its unipotent radical $\mathbb{U} \subset \mathbb{B}$, and a maximal torus $\mathbb{T} \subset \mathbb{B}$. Similarly, since $\alpha_i$ is preserved by $\sigma$ then the minuscule parabolic $P=LU_P$ descends to $\mathbb{P} = \mathbb{L} \mathbb{U}_P$ over $\mathbb{F}$. The same comments apply to the adjoint quotient $\mathbb{G}_{\ad}$. 

\begin{remark}
    It is possible that there is no $\sigma$-stable splitting $T \cong Z(G) \times T_{\ad}$, so we must now add this hypothesis to Assumption \ref{Assump1}. The splitting then descends to an isomorphism $\mathbb{T} \cong Z(\mathbb{G}) \times \mathbb{T}_{\ad}$ over $\mathbb{F}$.  We will show how to arrange the existence of such a splitting in the examples below, or one can work with groups of adjoint type.
\end{remark}

The coweight $\omega_i^\vee \in X_*(T_{\ad})$ is necessarily fixed by $\sigma$ since it is uniquely determined by pairings with roots, and $\alpha_i$ is fixed. Since we are assuming we have chosen a $\sigma$-stable splitting $T \cong Z(G) \times T_{\ad}$, then the induced lift $\tilde{\omega}_i^{\vee} \in X_*(T)$ is also fixed by $\sigma$, and hence it is defined over $\mathbb{F}$. This allows us to
descend the isomorphisms in Corollary \ref{cor:splitting-of-center-of-levi-part} and Lemma \ref{GmLevi} to isomorphisms over $\mathbb{F}$ of the form
$$Z(\mathbb{L}) \cong Z(\mathbb{G}) \times Z(\mathbb{L}_{\ad}) = Z(\mathbb{G}) \times \mathbb{G}_m.$$ 

As $\sigma$ preserves $\alpha_{i}$ and permutes our distinguished representatives of simple reflections, it follows that $\dot{w}_0$, $\dot{w}_0^P$, and $\dot{w}_P$ all descend to $\mathbb{G}(\mathbb{F})$. It is then straightforward to check that the geometric crystal $X = U Z(L) \dot{w}_P U \cap B_-$ descends to a scheme $\mathbb{X}$ over $\mathbb{F}$, as do all of the associated maps in Section \ref{sec:Crystal}.
The next result shows that Lemma \ref{lem-quot} is unchanged in the quasi-split case (in fact a more general statement is true for any rational parabolic of a quasi-split group, but we will not need this). 

\begin{lemma}
    We have  $$\frac{\sizeof{\left(\specialWeylElement{P}^{-1} \mathbb{U}(\mathbb{F}) \specialWeylElement{P}\right) \cap \mathbb{U}(\mathbb{F})}}{\sizeof{\mathbb{U}(\mathbb{F})}} = q^{-\dim G/P}.$$
\end{lemma}

\begin{proof}
    It is well-known (and straightforward to check) that the only automorphisms of Dynkin diagrams which give rise to non-reduced root systems occur in type $A_n$ where $n$ is even. Since we are not in this case, there are no roots $\alpha$ such that $\alpha+\sigma(\alpha)$ is also a root. Thus, if $\alpha \neq \sigma(\alpha)$ then $U_{\alpha} U_{\sigma(\alpha)}$ is a subgroup of $G$ isomorphic to $\mathbb{G}_a^2$. It follows that for every orbit $\overline{\alpha} = \{\alpha, \sigma(\alpha)\}$ of $\sigma$ on the roots of $G$, there is a corresponding root subgroup $\mathbb{U}_{\overline{\alpha}} \subset \mathbb{G}$ defined over $\mathbb{F}$ such that 
$$ \mathbb{U}_{\overline{\alpha}}(\mathbb{F}) \cong \begin{cases}
    \mathbb{F}, & \text{if } \sigma(\alpha)=\alpha \\
    \mathbb{E}, & \text{if } \sigma(\alpha) \neq \alpha.
\end{cases}$$
Revisiting the proof of Lemma \ref{lem-quot} now shows that the sizes of the groups in the statement of the lemma are the same as in the split case.
\end{proof}

The datum of $\psi \colon \mathbb{F}_p \rightarrow \mathbb{C}^{\times}$ is the same as for split groups, but now we take $\chi$ a character of $\mathbb{T}(\mathbb{F})$, along with its associated character sheaf $\K_{\chi}$ on $\mathbb{T}$.
We consider the induced representation $\Ind{\mathbb{B}_-(\mathbb{F})}{\mathbb{G}(\mathbb{F})}{\chi}$ as in Section \ref{sec:generic-representations-and-bessel-functions}, which by the same reasoning still contains a unique irreducible subrepresentation $\tau$ admitting a $\psi$-Whittaker vector. We may repeat Definition \ref{defWC} in this setup, giving us the weighted character sheaf $\WC_{\mathbb{G}}^{\mathbb{P}}(\phi, \chi)$ on $Z(\mathbb{L})$.
All of the results of Sections \ref{sec:generic-representations-and-bessel-functions} and \ref{sec-WC} go through unchanged. Thus, for any $z \in Z(\mathbb{L})(\mathbb{F})$, we have
    \begin{equation} \label{quasiTrace} \textnormal{Tr}_{{\WC_{\mathbb{G}}^{\mathbb{P}}}(\phi, \chi)}(z) = (-1)^{\dim G/P}q^{\tfrac{\dim G/P}{2}} \besselFunction_{\tau, \fieldCharacter}\left(z \specialWeylElement{P}\right).\end{equation}

It is straightforward to check that the affine Grassmannian $\Gr = \mathbb{G}(\mathbb{F}(\!(\tau)\!))/ \mathbb{G}(\mathbb{F}([\![\tau]\!])$ and $\Bun$ descend to $\mathbb{F}$, as do the Hecke stack $\Hk$ and all the associated maps from Section \ref{sec:K}. (In particular, we use that the sum of the trivializations of the simple roots,  the trivialization of the lowest root $-\theta$, and $\dot{\kappa}$ are all fixed by $\sigma$). If $\mu \in X_*(T)^+$ is fixed by $\sigma$ then $\Gr_{\mu}$ is defined over $\mathbb{F}$, so that 
the diagram \eqref{Hkdiagram} also is defined over $\mathbb{F}$. We may thus repeat Definition \ref{KDef}, giving us the Kloosterman sheaf $\text{Kl}_{\hat{\mathbb{G}}}^{\mu}(\phi, \chi)$ on $\mathbb{G}_m$ for any $\mu \in (X_*(T)^+)^\sigma$. (In comparison to \cite[\S 2.3]{HNY}, here we do not consider Kloosterman sheaves for more general representations of $^LG$.) The following corollary summarizes our results in the quasi-split case.

\begin{corollary} \label{QSplitCor}
Let $G$ be a pinned, split reductive group whose derived subgroup is almost simple and such that $Z(G)$ is a torus. Let $\sigma$ be a pinning-preserving automorphism which fixes the adjoint minuscule fundamental coweight $\omega_i^{\vee}$ associated to $\alpha_i$. Let $\mathbb{G} / \mathbb{F}$ be the quasi-split form of $G$ with respect to $\sigma$, and further suppose there exists a $\sigma$-equivariant isomorphism $T \cong Z(G) \times T_{\ad}$. Then for any nontrivial $\psi \colon \mathbb{F}_p \rightarrow \Ql^\times$ and any $\chi \colon \mathbb{T}(\mathbb{F}) \rightarrow \Ql^\times$, there is an isomorphism of local systems
    $$\WC_{\mathbb{G}}^{\mathbb{P}}(\phi, \chi)  \cong \K_{\chi}^{Z(\mathbb{G})} \boxtimes \alpha_i^*\textnormal{Kl}_{\hat{\mathbb{G}}}^{\tilde{\omega}_i^\vee}(\phi, \chi^{-1})$$
on $Z(\mathbb{L}) \cong Z(\mathbb{G}) \times Z(\mathbb{L}_{\ad})$.
Additionally, taking the trace of Frobenius for any $z = z_0 z_1 \in Z(\mathbb{L})(\mathbb{F})$ gives an equality $$\textnormal{Kl}_{\hat{\mathbb{G}}}^{\tilde{\omega}_i^\vee}(\phi, \chi^{-1}; \alpha_i(z))  = \overline{\chi(z_0)} (-1)^{\dim G/P} q^{\tfrac{\dim G/P}{2}} \besselFunction_{\tau, \fieldCharacter}\left(z \specialWeylElement{P}\right).$$
\end{corollary}

\begin{proof}
    This follows from \eqref{quasiTrace} and by checking that all of the steps in the proof of Theorem \ref{KLtoWC} are $\sigma$-equivariant. 
\end{proof}

The condition that there exists a $\sigma$-equivariant isomorphism $T \cong Z(G) \times T_{\ad}$ is not automatic, e.g., we will show in Section \ref{ref:QsplitA} that it fails for $\GL_n$ with the standard pinning and $T_{\ad}$ given by diagonal matrices with last entry $1$. If one has a $\sigma$-equivariant embedding $(G, T) \rightarrow (G', T')$ as in Lemma \ref{lemm-G'} then Remark \ref{notTorus} applies; this can be useful even if $Z(G)$ is a torus to arrange so that there is a $\sigma$-equivariant isomorphism $T' \cong Z(G') \times T_{\ad}$.

\section{Examples} \label{Sec:ExamplesSplit} 
In this section we explicitly determine the sets $Z(L) \dot{w}_P$ to which Corollaries \ref{maincor} and \ref{QSplitCor} apply, for all possible groups up to modifying the center with a few exceptions in characteristic 2. To generate the results in this section, we first used the computer algebra package PyCox \cite{PyCox} to compute reduced expressions for $w_0, w^P$, and $w_P$ for each type of Lie algebra and some small $n$. Then we used GAP3/Chevie \cite{GH96} to compute representatives for the simple reflections and multiply them. In all cases there appeared an obvious candidate for the matrices as a function of $n$. We then rigorously checked the pattern holds, with some suggestions from Gemini, Claude and ChatGPT. We generally omit these straightforward yet tedious arguments. In our examples, we first choose the associated homomorphism from $\SL_2$, and then our representatives are obtained by taking the images of the $2 \times 2$ matrix as in \eqref{simpleref}, in accordance with \cite{LamTemplier2024}. The interested reader can explore the $\SL_2$ embeddings and the long Weyl elements corresponding to different parabolic subgroups using the article's interactive companion programs, available at \cite{CassGuZelingher2026GitHub}.

Tables \ref{tab:minuscule_coweights_AC} and \ref{tab:minuscule_coweights_D} below give the elements $Z(L) \dot{w}_P$ in types A--D, and more details can be found in subsequent subsections; Type E is handled separately. After Type $A$ our explanations will be briefer in order to save space.
To read the tables below, one needs the following notation: $I_n$ is the $n \times n$ identity matrix, $J_n$ is the anti-diagonal $n \times n$ identity matrix, and \begin{equation} \label{Andef}
A_n = \text{diag}(1, -1, 1, \ldots, (-1)^{n-1}). \end{equation}
Additionally, the elements $c,t \in \mathbb{G}_m(\mathbb{F}) = \mathbb{F}^{\times}$ are arbitrary.

\begin{table}[htbp]
\centering
\caption{Special elements in Types A--C.}
\label{tab:minuscule_coweights_AC}
\renewcommand{\arraystretch}{1.5}
\begin{tabular}{|l|l|l|l|}
\hline
\textbf{Type} & \begin{tabular}[c]{@{}l@{}}\textbf{Minuscule} \\ \textbf{Coweight}\end{tabular} & $\boldsymbol{\dim G/P}$ & \textbf{Elements of} $\boldsymbol{Z(L) \dot{w}_P}$ \\ \hline
$A_{n-1}$     & $\omega_i^\vee$    & $i(n - i)$                & $\begin{pmatrix} 0 &  (-1)^{n-i} ct I_{i} \\  c I_{n-i} & 0 \end{pmatrix}$ \\ \hline
$B_n$         & $\omega_1^\vee$    & $2n-1$                    & $\begin{pmatrix} 0 & 0 & (-1)^n t \\ 0 & -I_{2n-1} & 0 \\ (-1)^n t^{-1} & 0 & 0 \end{pmatrix}$ \\ \hline
$C_n$         & $\omega_n^\vee$    & $\frac{n(n+1)}{2}$        & $\begin{pmatrix} 0 & (-1)^n ct A_n \\ cA_n & 0 \end{pmatrix}$ \\ \hline
\end{tabular}
\end{table}

\begin{table}[htbp]
\centering
\caption{Special elements in Type D.}
\label{tab:minuscule_coweights_D}
\renewcommand{\arraystretch}{1.5}
\begin{tabular}{|l|l|l|l|}
\hline
\textbf{Type} & \begin{tabular}[c]{@{}l@{}}\textbf{Minuscule} \\ \textbf{Coweight}\end{tabular} & $\boldsymbol{\dim G/P}$ & \textbf{Elements of} $\boldsymbol{Z(L) \dot{w}_P}$ \\ \hline
$D_n$         & $\omega_1^\vee$    & $2n-2$                    & $
 \begin{pmatrix}
    0 & 0 & 0 & 0 & (-1)^{n-1}ct^2 \\
    0 & -ctI_{n-2} & 0 & 0 & 0 \\
    0 & 0 & -ctJ_2 & 0 & 0 \\
    0 & 0 & 0 & -ctI_{n-2} & 0 \\
    (-1)^{n-1}c & 0 & 0 & 0 & 0
 \end{pmatrix}
$ \\ \hline
              & $\omega_{n-1}^\vee$ & $\frac{n(n-1)}{2}$       & $
\begin{gathered}
  \begin{pmatrix} 0 & 0 & ctA_{n-1} & 0  \\ c & 0 & 0 & 0  \\ 0 & 0 & 0 & ct  \\ 0 & -cA_{n-1} & 0 & 0 \end{pmatrix} \quad \text{if } n \text{ odd}  \\[1ex]
  \begin{pmatrix} 0 & 0 & -ct & 0 & 0 & 0 \\ 0 & 0& 0 & 0 & ctA_{n-2} & 0 \\ c & 0 & 0 & 0 & 0 & 0 \\ 0 & 0 & 0 & 0 & 0 & ct \\ 0 & -cA_{n-2} & 0 & 0 & 0 & 0 \\ 0 & 0 & 0 & -c & 0 & 0 \end{pmatrix} \quad \text{if } n \text{ even}
\end{gathered}
$ \\ \hline
              & $\omega_{n}^\vee$   & $\frac{n(n-1)}{2}$       & $
\begin{gathered}
  \begin{pmatrix}  0 & ct & 0 & 0  \\  0 & 0 & 0 & -ctA_{n-1} \\  cA_{n-1} & 0 & 0 & 0 \\ 0  & 0 & c & 0 \\  \end{pmatrix} \quad \text{if } n \text{ odd}  \\[1ex]
  \begin{pmatrix} 0 & -ctA_n \\ cA_n & 0 \end{pmatrix} \quad \text{if } n \text{ even} 
\end{gathered}
$ \\ \hline
\end{tabular}
\end{table}

\subsection{Type A} \label{sec-A} We start with $G = \GL_{n}$ for $n \geq 2$. We choose the standard upper triangular Borel $B$ and diagonal torus $T$. We choose the splitting $T \cong Z(G) \times T_{\ad}$ where $T_{\ad}$ consists of matrices of the form $\diag(t_1, \ldots, t_{n-1}, 1)$ for $t_i \in \mathbb{G}_m$. We denote the simple roots by $\alpha_i = \varepsilon_i - \varepsilon_{i+1}$ for $1 \leq i \leq n-1$. There is a canonical homomorphism $\varphi_{\alpha_i} \colon \SL_2 \rightarrow \GL_n$ which inserts the $2 \times 2$ matrix in rows and columns $i$ to $i+1$.
For $\alpha_i \in \Delta$ with simple reflection $s_i$, we choose \begin{equation} \label{simpleref} \dot{s}_i = \varphi_{\alpha_i} \begin{pmatrix} 0 & -1 \\ 1 & 0 \end{pmatrix}.\end{equation}
All of the fundamental coweights $\omega_i^\vee =\varepsilon_1^\vee + \cdots + \varepsilon_i^\vee$, $1 \leq i \leq n-1$, are minuscule, and their lifts from $\text{PGL}_n$ to $\GL_n$ induced by the above splitting are $$\omega_i^\vee(t) = \begin{pmatrix}
    t I_i & 0 \\ 0 & I_{n-i} \end{pmatrix} \quad 1 \leq i \leq n-1.$$ 
Let us fix $\omega_i^{\vee}$. One can check by induction on $n$ that
\begin{equation} \label{eqJ} \dot{w}_0 = \tilde{J}_n := \begin{pmatrix}
0 & \dots & 0 & (-1)^{n-1} \\
\vdots & \ddots & (-1)^{n-2} & 0 \\
0 & -1 & \ddots & \vdots \\
1 & 0 & \dots & 0
\end{pmatrix}, \quad \dot{w}_0^P = \begin{pmatrix} \tilde{J}_i & 0 \\ 0 & \tilde{J}_{n-i} \end{pmatrix}.\end{equation}
Then
$$\dot{w}_P = (\dot{w}_0^P)^{-1} \dot{w}_0 = \begin{pmatrix} 0 & (-1)^{n-i} I_{i} \\ I_{n-i} & 0 \end{pmatrix}.$$ An arbitrary element $c I_n \cdot \omega_i^\vee(t) \cdot \dot{w}_P$ of $Z(L)\dot{w}_P$ is then given as in Table \ref{tab:minuscule_coweights_AC}.

If $G = \text{PGL}_n$, the only difference is that the center is trivial, and the elements of $Z(L)\dot{w}_P$ are obtained as in Table \ref{tab:minuscule_coweights_AC}, taking $c=1$ and $t \in \mathbb{G}_m$. If $G = \text{SL}_n$, we may take $\text{SL}_n \rightarrow \GL_n$ for the embedding as in Lemma \ref{lemm-G'}. Then $Z(L)\dot{w}_P$ consists of those matrices as in Table \ref{tab:minuscule_coweights_AC} having determinant $1$, which amounts to requiring $c^nt^i = 1$. 

\subsection{Type B} Let $n \geq 2$. The group $G = \Sodd$ is of adjoint type. For convenience we avoid characteristic $2$ and work over $\mathbb{Z}[\frac{1}{2}]$, so that we may define $\Sodd \subset \text{SL}_{2n+1}$ consisting of those $g$ such that $g^TJ'g = J'$, where $J'$ is the anti-diagonal identity matrix modified with $1/2$ in the center, i.e. $$ J' = 
\text{anti-diag}(1, \ldots, 1, 1/2, 1, \ldots, 1).$$
The use of $J'$ instead of $J$ ensures that our representatives of simple reflections have entries in $\{-1, 0, 1\}$. We choose $B$ consisting of upper triangular matrices, and then $T$ consists of diagonal matrices of the form $\diag(t_1, \ldots, t_n, 1, t_{n}^{-1}, \ldots, t_1^{-1})$.
The simple roots are $\varepsilon_i - \varepsilon_{i+1}$ for $ 1 \leq i \leq n-1$ and $\varepsilon_n$. The standard homomorphisms $\SL_2 \rightarrow \Sodd$ for the first $n-1$ simple roots are given by mapping to two $2 \times 2$ blocks, e.g.~for $n=2$ and $i=1$ we have
$$\phi_{\varepsilon_1 - \varepsilon_{2}}\left( \begin{pmatrix} a & b \\ c & d \end{pmatrix} \right) = \begin{pmatrix} a & b & 0 & 0 & 0 \\ c & d & 0 & 0 & 0 \\ 0 & 0 & 1 & 0 & 0 \\ 0 & 0 & 0 & a & -b \\ 0 & 0 & 0 & -c & d \end{pmatrix}.$$
The homomorphism for $\varepsilon_n$ maps to the following $3 \times 3$ matrix in indices $n$ to $n+2$:
$$\begin{pmatrix} a & b \\ c & d \end{pmatrix} \mapsto \begin{pmatrix} a^2 & ab & -b^2 \\ 2ac & ad+bc & -2bd \\ -c^2 & -cd & d^2 \end{pmatrix}.$$
The unique minuscule fundamental coweight $\omega_1^{\vee} = \varepsilon_1^\vee$ is given by $$\omega_1^{\vee}(t) = \begin{pmatrix} t & 0 & 0 \\ 0 & I_{2n-1} & 0 \\ 0 & 0 & t^{-1} \end{pmatrix}.$$

\subsection{Type C} Let $n \geq 2$, and 
let $\Spn \subset \SL_{2n}$ be the subgroup consisting of those $g$ such that $$g^T \Omega g = \Omega, \quad \text{ where } \:\Omega = \begin{pmatrix}
    0 & J_n \\ -J_n & 0
\end{pmatrix}.$$ Here $J_n$ is the anti-diagonal $n \times n$ identity matrix. 
The center of $\Spn$ is $\mu_2$, so we consider the group $G = \GSpn \subset \GL_{2n}$ consisting of those $g$ such that $g^T \Omega g = c \Omega$ for some $c \in \mathbb{G}_m$. We let $B \subset \GSpn$ be upper triangular matrices, and then $T$ consists of diagonal matrices of the form $\diag(t_1, \ldots, t_n, ct_{n}^{-1}, \ldots, ct_1^{-1})$ for $t_i, c \in \mathbb{G}_m$. The inclusion of the maximal torus of $\Spn$ into $T$ splits by sending $\diag(t_1, \ldots, t_n, ct_{n}^{-1}, \ldots, ct_1^{-1})$ to $\diag(t_1, \ldots, t_n, t_{n}^{-1}, \ldots, t_1^{-1})$.
The simple roots of $\Spn$ are $\varepsilon_i - \varepsilon_{i+1}$ for $ 1 \leq i \leq n-1$ and $2\varepsilon_n$. The standard homomorphisms $\SL_2 \rightarrow \Spn$ for the first $n-1$ simple roots are given by mapping to two $2 \times 2$ blocks, and the homomorphism for $2\varepsilon_n$ is obtained by mapping to a single copy of $\SL_2$ located in the center. For example, when $n=2$ these homomorphisms are
\begin{equation} \label{Choms} \phi_{\varepsilon_1 - \varepsilon_{2}}\left( \begin{pmatrix} a & b \\ c & d \end{pmatrix} \right) = \begin{pmatrix} a & b & 0 & 0 \\ c & d & 0 & 0 \\ 0 & 0 & a & -b \\ 0 & 0 & -c & d \end{pmatrix}, \quad \phi_{2\varepsilon_{2}}\left( \begin{pmatrix} a & b \\ c & d \end{pmatrix} \right) = \begin{pmatrix} 1 & 0 & 0 & 0 \\ 0 & a & b & 0 \\ 0 & c & d & 0 \\ 0 & 0 & 0 & 1 \end{pmatrix}.\end{equation}
We choose the splitting $T \cong Z(G) \times T_{\ad}$ uniquely determined by the requirement that all matrices in $T_{\ad}$ have a $1$ in the bottom right entry. 
The lift to $G$ of the unique minuscule fundamental coweight $\omega_n^\vee = \frac{1}{2}(\varepsilon_1^\vee + \cdots + \varepsilon_n^\vee)$ is $$\omega_n^{\vee}(t) = \begin{pmatrix} t I_n & 0 \\ 0 & I_n \end{pmatrix}.$$

\subsection{Type D} \label{sec:D} We work over $\mathbb{Z}[\frac{1}{2}]$. Let $n \geq 4$, and let
$\SO_{2n} \subset \SL_{2n}$ be the subgroup consisting of those $g$ such that $g^TJg = J$ where $J=J_{2n}$ is the anti-diagonal identity matrix. As with $\Sp_{2n}$, the center is $\mu_2$ so we let $\mathrm{GO}_{2n}\subset \GL_{2n}$ be the subgroup preserving $J$ up to a scaling factor. The group $\mathrm{GO}_{2n}$ is disconnected, so we let $G = \GSO_{2n}$ be the connected component of the identity. We take the Borel to be the upper triangular matrices, and then $T$ consists of diagonal matrices of the form $\diag(t_1, \ldots, t_n, ct_{n}^{-1}, \ldots, ct_1^{-1})$ for $t_i, c \in \mathbb{G}_m$. We again choose the splitting $T \cong Z(G) \times T_{\ad}$ such that all matrices in $T_{\ad}$ have a $1$ in the bottom right entry.
The simple roots of $\SO_{2n}$ are $\varepsilon_i - \varepsilon_{i+1}$ for $ 1 \leq i \leq n-1$ and $\varepsilon_{n-1} + \varepsilon_{n}$. The fundamental coweights are $\omega_i^\vee = \varepsilon_1^{\vee} + \cdots + \varepsilon_i^{\vee}$ for $1 \leq i \leq n-2$, $\omega_{n-1}^\vee =  \frac{1}{2}(\varepsilon_1^{\vee}+\cdots + \varepsilon_{n-1}^{\vee} -\varepsilon_{n}^{\vee})$, and $\omega_n^\vee = \frac{1}{2}(\varepsilon_1^{\vee}+\cdots +\varepsilon_{n-1}^{\vee} +\varepsilon_{n}^{\vee})$. The minuscule coweights are $\omega_1^{\vee}$, $\omega_{n-1}^{\vee}$, and $\omega_n^{\vee}$.
For the first $n-1$ simple roots we take the same homomorphisms from $\SL_2$ as for $\Sp_{2n}$, e.g.~as in the left side of \eqref{Choms}. For the last simple root, we use the following $4 \times 4$ matrix from the $n=2$ case, embedded in the center of the identity matrix:
$$\phi_{\varepsilon_{1}+\varepsilon_{2}}\left( \begin{pmatrix} a & b \\ c & d \end{pmatrix} \right) = \begin{pmatrix} a & 0 & b & 0 \\ 0 & a & 0 & -b \\ c & 0 & d & 0 \\ 0 & -c & 0 & d \end{pmatrix}.$$
We have \begin{equation} \label{omega1eq} \omega_1^{\vee}(t) = \begin{pmatrix}
    t^2 & 0 & 0 \\ 0 & tI_{2n-2} & 0 \\ 0 & 0 & 1
\end{pmatrix}, \quad \omega_{n-1}^{\vee}(t) = \begin{pmatrix}
    tI_{n-1} & 0 & 0& 0\\
    0 & 1 & 0 & 0 \\
    0 & 0 & t & 0 \\
    0 & 0 & 0 & I_{n-1}
\end{pmatrix}, \quad \omega_n^{\vee}(t) = \begin{pmatrix}
    t I_n & 0 \\ 0 & I_n
\end{pmatrix}.
\end{equation}
A general element of $Z(L)\dot{w}_P$ is then given as in Table \ref{tab:minuscule_coweights_D}.

\subsection{Type E} \label{sec:E} We use the following Cartan matrix for $\Esev$ (resp.~the upper left $6 \times 6$ minor for $\Esix$).
$$\begin{pmatrix}
2 & 0 & -1 & 0 & 0 & 0 &0 \\
0 & 2 & 0 & -1 & 0 & 0 & 0 \\
-1 & 0 & 2 & -1 & 0 & 0 & 0\\
0 & -1 & -1 & 2 & -1 & 0 & 0\\
0 & 0 & 0 & -1 & 2 & -1 & 0\\
0 & 0 & 0 & 0 & -1 & 2 & -1 \\
0 & 0 & 0 & 0 & 0 & -1 & 2
\end{pmatrix}.$$
We denote the simple roots, fundamental weights, and fundamental coweights by $\alpha_i$, $\omega_i$, and $\omega_i^{\vee}$ for $i=1, \ldots, 7$ (resp. $i=1, \ldots, 6$). The minuscule fundamental coweight is $\omega_7^\vee$ for $\Esev$ (resp.~$\omega_1^\vee$ and~$\omega_6^\vee$ for $\Esix$). We realize the Chevalley group $\mathrm{E}_7$ (resp.~$\mathrm{E}_6$) as a subgroup of $\GL_{56}$ (resp.~$\GL_{27}$) via the minuscule representation of highest weight $\omega_7$ (resp.~$\omega_1$) and the canonical basis in \cite[\S4]{Geck17}. These bases consist of weight vectors, so the only choices made in loc.~cit.~are orderings and signs. For the convenience of the reader we record the ordering below. To save space we do not record the associated homomorphisms from $\SL_2$ for the simple roots, but we note that they contain no negative signs and the resulting Borel subgroups are upper triangular. 
\\ \\
\noindent \textbf{$\boldsymbol{\Esix}$}. The standard basis of $\mathbb{F}^{\oplus 27}$ consists of weight vectors for a maximal torus of $\Esix$, ordered in the following way and written as linear combinations of $\omega_1, \ldots, \omega_6$:

\allowdisplaybreaks

\begin{align*}
    &(1, 0, 0, 0, 0, 0), \; (-1, 0, 1, 0, 0, 0), \; (0, 0, -1, 1, 0, 0), \; (0, 1, 0, -1, 1, 0), \; (0, -1, 0, 0, 1, 0), \\
    &(0, 1, 0, 0, -1, 1), \; (0, -1, 0, 1, -1, 1), \; (0, 1, 0, 0, 0, -1), \; (0, 0, 1, -1, 0, 1), \; (0, -1, 0, 1, 0, -1), \\
    &(1, 0, -1, 0, 0, 1), \; (0, 0, 1, -1, 1, -1), \; (-1, 0, 0, 0, 0, 1), \; (1, 0, -1, 0, 1, -1), \; (0, 0, 1, 0, -1, 0), \\
    &(1, 0, -1, 1, -1, 0), \; (-1, 0, 0, 0, 1, -1), \; (1, 1, 0, -1, 0, 0), \; (-1, 0, 0, 1, -1, 0), \; (1, -1, 0, 0, 0, 0), \\
    &(-1, 1, 1, -1, 0, 0), \; (-1, -1, 1, 0, 0, 0), \; (0, 1, -1, 0, 0, 0), \; (0, -1, -1, 1, 0, 0), \; (0, 0, 0, -1, 1, 0), \\
    &(0, 0, 0, 0, -1, 1), \; (0, 0, 0, 0, 0, -1).
\end{align*}

The simply connected group $\Esix$ has center $\mu_3$, so we consider the group $G = \Gsix \subset \GL_{27}$ generated by $E_6$ and scalar matrices. We again choose the splitting $T \cong Z(G) \times T_{\ad}$ such that all matrices in $T_{\ad}$ have a $1$ in the bottom right entry. We omit the homomorphisms from $\SL_2$ to save space. The lifts to $G$ of the minuscule fundamental coweights are
\begin{align*}
\omega_1^\vee(t) &= \diag(t^2,\ \underbrace{t,\dots,t}_{11},\ 1,\ t,\ t,\ t,\ 1,\ t,\ 1,\ t,\ \underbrace{1,\dots,1}_{7}), \\
\omega_6^\vee(t) &= \diag(\underbrace{t^2,\dots,t^2}_{7},\ t,\ t^2,\ t,\ t^2,\ t,\ t^2,\ \underbrace{t,\dots,t}_{13},\ 1).
\end{align*}

In both cases $\dim G /P = 16$, and as an element of $\GL_{27}(\mathbb{F})$ the matrix $\dot{w}_P$ is a signed permutation matrix, where the underlying permutation is a product of the following nine disjoint 3-cycles:
\begin{align*}
\dot{w}_P \text{ for }\omega_1^\vee \: \colon \: & (1\,13\,27)(2\,17\,8)(3\,19\,10)(4\,21\,12)(5\,23\,15)(6\,22\,14)(7\,24\,16)(9\,25\,18)(11\,26\,20),  \\
\dot{w}_P \text{ for }\omega_6^\vee \: \colon \: & (1\,27\,13)(2\,8\,17)(3\,10\,19)(4\,12\,21)(5\,15\,23)(6\,14\,22)(7\,16\,24)(9\,18\,25)(11\,20\,26).
\end{align*}
These two permutations are inverses, which is explained by the fact that $w_0$ swaps $\Delta \setminus \{\alpha_1\}$ and $\Delta \setminus \{\alpha_6\}$. For $\omega_1^\vee$, the signs of $\dot{w}_P$ are $+1$ in columns $\{1,2,3,4,5,6,7,9,11,13,27\}$, and $-1$ otherwise. For $\omega_6^\vee$, the signs of $\dot{w}_P$ are $+1$ in columns $\{1,13,17,19,21,22,23,24,25,26,27\}$, and $-1$ otherwise. An arbitrary element of $Z(L) \dot{w}_P$ is then given by $cI_{27} \cdot \omega_i^\vee(t) \cdot \dot{w}_P$ for $c, t \in \mathbb{F}^\times$, which can easily be made explicit given the above information.
\\ \\
\noindent \textbf{$\boldsymbol{\Esev}$}. The standard basis of $\mathbb{F}^{\oplus 56}$ consists of weight vectors for a maximal torus of $\Esev$, ordered in the following way and written as linear combinations of $\omega_1, \ldots, \omega_7$:

\begin{align*}
    &(0, 0, 0, 0, 0, 0, 1), \; (0, 0, 0, 0, 0, 1, -1), \; (0, 0, 0, 0, 1, -1, 0), \; (0, 0, 0, 1, -1, 0, 0), \\
    &(0, 1, 1, -1, 0, 0, 0), \; (0, -1, 1, 0, 0, 0, 0), \; (1, 1, -1, 0, 0, 0, 0), \; (1, -1, -1, 1, 0, 0, 0), \\
    &(-1, 1, 0, 0, 0, 0, 0), \; (-1, -1, 0, 1, 0, 0, 0), \; (1, 0, 0, -1, 1, 0, 0), \; (-1, 0, 1, -1, 1, 0, 0), \\
    &(1, 0, 0, 0, -1, 1, 0), \; (0, 0, -1, 0, 1, 0, 0), \; (-1, 0, 1, 0, -1, 1, 0), \; (1, 0, 0, 0, 0, -1, 1), \\
    &(0, 0, -1, 1, -1, 1, 0), \; (-1, 0, 1, 0, 0, -1, 1), \; (1, 0, 0, 0, 0, 0, -1), \; (0, 1, 0, -1, 0, 1, 0), \\
    &(0, 0, -1, 1, 0, -1, 1), \; (-1, 0, 1, 0, 0, 0, -1), \; (0, -1, 0, 0, 0, 1, 0), \; (0, 1, 0, -1, 1, -1, 1), \\
    &(0, 0, -1, 1, 0, 0, -1), \; (0, -1, 0, 0, 1, -1, 1), \; (0, 1, 0, 0, -1, 0, 1), \; (0, 1, 0, -1, 1, 0, -1), \\
    &(0, -1, 0, 1, -1, 0, 1), \; (0, -1, 0, 0, 1, 0, -1), \; (0, 1, 0, 0, -1, 1, -1), \; (0, 0, 1, -1, 0, 0, 1), \\
    &(0, -1, 0, 1, -1, 1, -1), \; (0, 1, 0, 0, 0, -1, 0), \; (1, 0, -1, 0, 0, 0, 1), \; (0, 0, 1, -1, 0, 1, -1), \\
    &(0, -1, 0, 1, 0, -1, 0), \; (-1, 0, 0, 0, 0, 0, 1), \; (1, 0, -1, 0, 0, 1, -1), \; (0, 0, 1, -1, 1, -1, 0), \\
    &(-1, 0, 0, 0, 0, 1, -1), \; (1, 0, -1, 0, 1, -1, 0), \; (0, 0, 1, 0, -1, 0, 0), \; (-1, 0, 0, 0, 1, -1, 0), \\
    &(1, 0, -1, 1, -1, 0, 0), \; (-1, 0, 0, 1, -1, 0, 0), \; (1, 1, 0, -1, 0, 0, 0), \; (1, -1, 0, 0, 0, 0, 0), \\
    &(-1, 1, 1, -1, 0, 0, 0), \; (-1, -1, 1, 0, 0, 0, 0), \; (0, 1, -1, 0, 0, 0, 0), \; (0, -1, -1, 1, 0, 0, 0), \\
    &(0, 0, 0, -1, 1, 0, 0), \; (0, 0, 0, 0, -1, 1, 0), \; (0, 0, 0, 0, 0, -1, 1), \; (0, 0, 0, 0, 0, 0, -1).
\end{align*}

The simply connected group $\Esev$ has center $\mu_2$, so we consider the group $G = \Gsev \subset \GL_{56}$ generated by $\Esev$ and scalar matrices. The lift of the minuscule fundamental coweight is 
$$
\omega_7^\vee(t) = \text{diag}\Big( t^3, \underbrace{t^2, \dots, t^2}_{17}, t, t^2, t^2, t, t^2, t^2, t, t^2, t^2, t, t^2, t, t, t^2, t, t, t^2, t, t, t^2, \underbrace{t, \dots, t}_{17}, 1).
$$

We have $\dim G/P = 27$, and as an element of $\GL_{56}(\mathbb{F})$ the matrix $\dot{w}_P$ is a signed permutation matrix, where the underlying permutation is a product of the following 28 disjoint 2-cycles:
$$\begin{aligned} \dot{w}_P \text{ for }\omega_7^\vee \: \colon \:
&(1\ 56)(2\ 19)(3\ 22)(4\ 25)(5\ 28)(6\ 30)(7\ 31)(8\ 33)(9\ 34)(10\ 37) \\
&(11\ 36)(12\ 40)(13\ 39)(14\ 43)(15\ 42)(16\ 41)(17\ 45)(18\ 44)(20\ 47) \\
&(21\ 46)(23\ 48)(24\ 49)(26\ 50)(27\ 51)(29\ 52)(32\ 53)(35\ 54)(38\ 55).
\end{aligned}$$
The signs are $-1$ in columns
$$\{2, \ldots,18 \} \cup \{20,21,23,24,26,27,29,32,35,38,56\},
$$
and $+1$ otherwise. An arbitrary element of $Z(L) \dot{w}_P$ is then given by $cI_{56} \cdot \omega_7^\vee(t) \cdot \dot{w}_P$ for $c, t \in \mathbb{F}^\times$, which can easily be made explicit given the above information.

\subsection{Quasi-split Type A} \label{ref:QsplitA}
Let $n$ be even, and consider the group $G = \GL_n \times \mathbb{G}_m$ (so we are in type $A_{n-1}$). We fix the pinning of this group corresponding to the standard pinning of $\GL_n$ as in Section \ref{sec-A}, with Borel $B \times \mathbb{G}_m$. Let $\tilde{J} = \dot{w}_0$ be the $n \times n$ alternating anti-diagonal identity matrix as in \eqref{eqJ}. Note that since $n$ is even, $\tilde{J}^{-1} = \tilde{J}^T = -\tilde{J}$. Consider the involution of $\GL_n \times \mathbb{G}_m$ given by $\sigma(g,  \lambda) = (\lambda \tilde{J} (g^{-1})^T \tilde{J}^{-1}, \lambda)$. Since $\tilde{J} = \dot{w}_0$, the involution $\sigma$ preserves the pinning of $B \times \mathbb{G}_m$ (the alternating signs in $\tilde{J}$ are needed so that the trivializations of the simple root groups are preserved). 

Restricted to the maximal torus $T \times \mathbb{G}_m$, we have
\begin{equation} \label{eq-GUT} \sigma( \diag(t_1, t_2, \ldots, t_n), \lambda) = (\diag(\lambda t_n^{-1}, \ldots, \lambda t_2^{-1}, \lambda t_1^{-1}), \lambda). \end{equation}
It is straightforward to check that $\sigma$ realizes the unique nontrivial automorphism of the Dynkin diagram of type $A_{n-1}$. The group $\mathbb{G}$ defined by \eqref{Geq} is the general unitary group $\GU$, and its defining equations can also be written as
\begin{equation} \label{GUeq} \GU(R) = \{ (g, \lambda) \in (\GL_n \times \mathbb{G}_m)(\mathbb{E} \otimes_{\mathbb{F}} R) \: \mid \: \overline{g}\tilde{J} g^T = \  \overline{\lambda} \tilde{J}, \: \overline{\lambda} = \lambda\}.\end{equation}
Then $\GU(\mathbb{F})$ consists of $g \in \GL_n(\mathbb{E})$ such that $\overline{g} \tilde{J} g^T = \  \lambda \tilde{J}$ for some $\lambda \in \mathbb{F}^{\times}$.

Note that $Z(\GL_n \times \mathbb{G}_m) = Z(\GL_n) \times \mathbb{G}_m$ consists of elements of the form $\diag(t, t, \ldots, t, \lambda)$. 
In Section \ref{sec-A} we chose the isomorphism $T \cong Z(\GL_n) \times T_{\ad}$ where $T_{\ad} \subset \GL_n$ consists of diagonal matrices with last entry equal to $1$. Similarly, we define a splitting of the adjoint quotient $T \times \mathbb{G}_m \rightarrow (T \times \mathbb{G}_m)_{\ad}$ by viewing 
$(T \times \mathbb{G}_m)_{\ad} \subset T \times \mathbb{G}_m$ as elements of the form $\diag(t_1, t_2, \ldots, t_{n-1}, 1, t_1)$. By \eqref{eq-GUT}, $(T \times \mathbb{G}_m)_{\ad}$ is preserved by $\sigma$, so by using the same equations as in \eqref{GUeq} we may define an isomorphism $\mathbb{T} \cong Z(\GU) \times \mathbb{T}_{\ad}$ over $\mathbb{F}$ (without the extra $\mathbb{G}_m$ factor, such an isomorphism would not exist). We have
\begin{align*}
    \mathbb{T}(\mathbb{F}) &= \{ (\diag(t_1, \ldots, t_{n/2}, \lambda t_{n/2}^{-q}, \ldots, \lambda t_{1}^{-q}), \lambda) \: \mid \: \lambda \in \mathbb{F}^\times, \: t_1, \ldots, t_{n/2} \in \mathbb{E}^\times \} \\ & \cong  (\mathbb{E}^{\times})^{n/2} \times \mathbb{F}^{\times}. \end{align*}
With respect to the isomorphism $\mathbb{T} \cong Z(\GU) \times \mathbb{T}_{\ad}$, the subgroup $\mathbb{T}_{\ad}(\mathbb{F})$ is given by setting $t_1 = \lambda \in \mathbb{F}^\times$ above, and $Z(\GU)(\mathbb{F})$ is given by setting $t_i = t \in \mathbb{E}^{\times}$ for all $i$, where $t^{q+1} = \lambda$.

Let $P \times \mathbb{G}_m = (L \times \mathbb{G}_m) (U_P \times \{1\}) \subset \GL_n \times \mathbb{G}_m$ be the parabolic defined by removing the simple root $\alpha_{n/2}$. These descend to groups $\mathbb{P} = \mathbb{L} \mathbb{U}_P \subset \GU$. The induced lift to $T \times \mathbb{G}_m$ of the minuscule fundamental coweight corresponding to $\alpha_{n/2}$ is
$$\omega_{n/2}^{\vee}(t) = (\diag(t, \ldots, t, 1, \ldots, 1), t),$$ where there are $n/2$ copies of $t$ in the first block. Then using Table \ref{tab:minuscule_coweights_AC}, an arbitrary element of $Z(\mathbb{L})(\mathbb{F})\dot{w}_P$ may be written
$$\begin{pmatrix} 0 & (-1)^{n/2}ct I_{n/2} \\ cI_{n/2} & 0\end{pmatrix}, \quad c \in \mathbb{E}^{\times}, \: t \in \mathbb{F}^{\times}.$$
This gives all of the setup necessary to apply Corollary \ref{QSplitCor}; we also note that $\dim \GL_n/P = \frac{n^2}{4}$ is even if and only if $n \equiv 0 \pmod{4}$.

\subsection{Quasi-split Type D}
Suppose that $n \geq 4$ and that we are in case (\ref{case3}) as in Section \ref{Sec:Qsplit}, so that the fixed root is $\alpha_1 = \varepsilon_1 - \varepsilon_2$ (we leave case (\ref{case2}) to the reader). As in Section \ref{sec:D}, we let $G = \GSO_{2n} \subset \GL_{2n}$, and we use the pinning defined there.
The involution $\sigma$ should send $\varepsilon_n$ to $-\varepsilon_n$ and fix $\varepsilon_i$ for $1 \leq i \leq n-1$. It is not hard to check that we may take $\sigma(g) = \Sigma g \Sigma^{-1}$ where $$\Sigma = \begin{pmatrix} I_{n-1} & 0 & 0 & 0\\ 0 &0 & 1 & 0 \\ 0 & 1 & 0 & 0 \\ 0 & 0 & 0 & I_{n-1}\end{pmatrix}.$$
We have
$$\sigma (\diag(t_1, \ldots, t_n, ct_{n}^{-1}, \ldots, ct_1^{-1})) = \diag(t_1, \ldots, t_{n-1}, ct_n^{-1}, t_n, ct_{n-1}^{-1}, \ldots, ct_1^{-1})).$$
The center $Z(G)$ is obtained by setting $t_i=t$ for all $i$ and $c=t^2$. We define $T_{\ad}$ by setting $c=t_1$ (so the last entry is $1$), and then we get a $\sigma$-equivariant isomorphism $T \cong Z(G) \times T_{\ad}$. We define $\mathbb{G}$ as in \eqref{Geq}, and then
\begin{align*} \mathbb{T}(\mathbb{F}) &= \{\diag(t_1, \ldots, t_n, ct_{n}^{-1}, \ldots, ct_1^{-1}) \: \mid \: c, t_1, \ldots, t_{n-1} \in \mathbb{F}^{\times}, \: t_n \in \mathbb{E}^{\times}, \: t_n^{q+1} = c \} \\
& \cong (\mathbb{F}^\times)^{n-1} \times \mathbb{E}^\times.
\end{align*}
The lift of $\omega_1^{\vee}$ to $G$ is given exactly as in \eqref{omega1eq}, and an arbitrary element of $Z(\mathbb{L})(\mathbb{F})\dot{w}_P$ is given exactly as in the first row of Table \ref{tab:minuscule_coweights_D}, where $c$, $t \in \mathbb{F}^{\times}$. We also note that $\dim G/P = 2n-2$ is even, so the factor $(-1)^{\dim G/P}$ in Corollary \ref{QSplitCor} is always equal to $1$.

\bibliographystyle{alpha}
\bibliography{references}

@Article{AlperinJames1995,
	author     = {Alperin, J. L. and James, G. D.},
	journal    = {J. Algebra},
	title      = {Bessel functions on finite groups},
	year       = {1995},
	issn       = {0021-8693,1090-266X},
	number     = {2},
	pages      = {524--530},
	volume     = {171},
	doi        = {10.1006/jabr.1995.1025},
	fjournal   = {Journal of Algebra},
	mrclass    = {20C15},
	mrnumber   = {1315910},
	mrreviewer = {I.\ M.\ Isaacs},
	url        = {https://doi.org/10.1006/jabr.1995.1025},
}

@article {LiuZhang2022,
	AUTHOR = {Liu, Baiying and Zhang, Qing},
	TITLE = {Gamma factors and converse theorems for classical groups over
	finite fields},
	JOURNAL = {J. Number Theory},
	FJOURNAL = {Journal of Number Theory},
	VOLUME = {234},
	YEAR = {2022},
	PAGES = {285--332},
	ISSN = {0022-314X,1096-1658},
	MRCLASS = {20C33 (20G40)},
	MRNUMBER = {4370538},
	MRREVIEWER = {Zhe\ Chen},
	DOI = {10.1016/j.jnt.2021.06.024},
	URL = {https://doi.org/10.1016/j.jnt.2021.06.024},
}

@article {LiuZhang2022b,
	AUTHOR = {Liu, Baiying and Zhang, Qing},
	TITLE = {On a converse theorem for {$\rm G_2$} over finite fields},
	JOURNAL = {Math. Ann.},
	FJOURNAL = {Mathematische Annalen},
	VOLUME = {383},
	YEAR = {2022},
	NUMBER = {3-4},
	PAGES = {1217--1283},
	ISSN = {0025-5831,1432-1807},
	MRCLASS = {20C33 (20G40)},
	MRNUMBER = {4458400},
	MRREVIEWER = {Rongqing\ Ye},
	DOI = {10.1007/s00208-021-02250-2},
	URL = {https://doi.org/10.1007/s00208-021-02250-2},
}

@article {LiuHazeltine2024,
	AUTHOR = {Hazeltine, Alexander and Liu, Bai Ying},
	TITLE = {A converse theorem for split {${\rm SO}_{2l}$} over finite
	fields},
	JOURNAL = {Acta Math. Sin. (Engl. Ser.)},
	FJOURNAL = {Acta Mathematica Sinica (English Series)},
	VOLUME = {40},
	YEAR = {2024},
	NUMBER = {3},
	PAGES = {731--771},
	ISSN = {1439-8516,1439-7617},
	MRCLASS = {20C33 (20G40)},
	MRNUMBER = {4712830},
	MRREVIEWER = {Damiano\ Rossi},
	DOI = {10.1007/s10114-023-2061-6},
	URL = {https://doi.org/10.1007/s10114-023-2061-6},
}

@Article{LamTemplier2024,
  author   = {Lam, Thomas and Templier, Nicolas},
  journal  = {Duke Math. J.},
  title    = {The mirror conjecture for minuscule flag varieties},
  year     = {2024},
  issn     = {0012-7094,1547-7398},
  number   = {1},
  pages    = {75--175},
  volume   = {173},
  doi      = {10.1215/00127094-2024-0007},
  fjournal = {Duke Mathematical Journal},
  mrclass  = {14D24 (11T23 14J33 14M15)},
  mrnumber = {4728689},
  url      = {https://doi.org/10.1215/00127094-2024-0007},
}

@article {PyCox,
    AUTHOR = {Geck, Meinolf},
     TITLE = {{PyCox}: computing with (finite) {C}oxeter groups and
              {I}wahori-{H}ecke algebras},
   JOURNAL = {LMS J. Comput. Math.},
  FJOURNAL = {LMS Journal of Computation and Mathematics},
    VOLUME = {15},
      YEAR = {2012},
     PAGES = {231--256},
      ISSN = {1461-1570},
   MRCLASS = {20C40 (20C08 20F55)},
  MRNUMBER = {2988815},
MRREVIEWER = {Chi\ Kin\ Mak},
       DOI = {10.1112/S1461157012001064},
       URL = {https://doi-org.ccl.idm.oclc.org/10.1112/S1461157012001064},
}

@book {Kac90,
    AUTHOR = {Kac, Victor G.},
     TITLE = {Infinite-dimensional {L}ie algebras},
   EDITION = {Third},
 PUBLISHER = {Cambridge University Press, Cambridge},
      YEAR = {1990},
     PAGES = {xxii+400},
      ISBN = {0-521-37215-1; 0-521-46693-8},
   MRCLASS = {17B65 (17B67 17B68 58F07)},
  MRNUMBER = {1104219},
       DOI = {10.1017/CBO9780511626234},
       URL = {https://doi-org.ccl.idm.oclc.org/10.1017/CBO9780511626234},
}

@article {NP01,
    AUTHOR = {Ng\^{o}, B. C. and Polo, P.},
     TITLE = {R\'{e}solutions de {D}emazure affines et formule de
              {C}asselman-{S}halika g\'{e}om\'{e}trique},
   JOURNAL = {J. Algebraic Geom.},
  FJOURNAL = {Journal of Algebraic Geometry},
    VOLUME = {10},
      YEAR = {2001},
    NUMBER = {3},
     PAGES = {515--547},
      ISSN = {1056-3911},
   MRCLASS = {14F43 (14F17)},
  MRNUMBER = {1832331},
MRREVIEWER = {Nguy\cftil{e}n Qu\^{o}c Th\'{a}ng},
}

@incollection {BK07,
    AUTHOR = {Berenstein, Arkady and Kazhdan, David},
     TITLE = {Geometric and unipotent crystals. {II}. {F}rom unipotent
              bicrystals to crystal bases},
 BOOKTITLE = {Quantum groups},
    SERIES = {Contemp. Math.},
    VOLUME = {433},
     PAGES = {13--88},
 PUBLISHER = {Amer. Math. Soc., Providence, RI},
      YEAR = {2007},
      ISBN = {978-0-8218-3713-9},
   MRCLASS = {17B37 (14L30 22E47)},
  MRNUMBER = {2349617},
MRREVIEWER = {Jonathan\ Brundan},
       DOI = {10.1090/conm/433/08321},
       URL = {https://doi.org/10.1090/conm/433/08321},
}

@incollection {functionsheaf,
    AUTHOR = {Deligne, P.},
     TITLE = {Rapport sur la formule des traces},
 BOOKTITLE = {Cohomologie \'etale},
    SERIES = {Lecture Notes in Math.},
    VOLUME = {569},
     PAGES = {76--109},
 PUBLISHER = {Springer, Berlin},
      YEAR = {1977},
      ISBN = {3-540-08066-X; 0-387-08066-X},
   MRCLASS = {14F20 (11G40 14G10)},
  MRNUMBER = {3727434},
       DOI = {10.1007/BFb0091519},
       URL = {https://doi-org.ccl.idm.oclc.org/10.1007/BFb0091519},
}

@article {MRR23,
    AUTHOR = {Mayeux, Arnaud and Richarz, Timo and Romagny, Matthieu},
     TITLE = {N\'eron blowups and low-degree cohomological applications},
   JOURNAL = {Algebr. Geom.},
  FJOURNAL = {Algebraic Geometry},
    VOLUME = {10},
      YEAR = {2023},
    NUMBER = {6},
     PAGES = {729--753},
      ISSN = {2313-1691,2214-2584},
   MRCLASS = {14L15},
  MRNUMBER = {4673395},
MRREVIEWER = {Alan\ Koch},
       DOI = {10.14231/ag-2023-026},
       URL = {https://doi.org/10.14231/ag-2023-026},
}

@incollection {ZhuAffine,
    AUTHOR = {Zhu, Xinwen},
     TITLE = {An introduction to affine {G}rassmannians and the geometric
              {S}atake equivalence},
 BOOKTITLE = {Geometry of moduli spaces and representation theory},
    SERIES = {IAS/Park City Math. Ser.},
    VOLUME = {24},
     PAGES = {59--154},
 PUBLISHER = {Amer. Math. Soc., Providence, RI},
      YEAR = {2017},
      ISBN = {978-1-4704-3574-5},
   MRCLASS = {14M15 (14D24 20F65 22E57)},
  MRNUMBER = {3752460},
MRREVIEWER = {Felipe\ Zald\'ivar},
}

@article {XZ22,
    AUTHOR = {Xu, Daxin and Zhu, Xinwen},
     TITLE = {Bessel {$F$}-isocrystals for reductive groups},
   JOURNAL = {Invent. Math.},
  FJOURNAL = {Inventiones Mathematicae},
    VOLUME = {227},
      YEAR = {2022},
    NUMBER = {3},
     PAGES = {997--1092},
      ISSN = {0020-9910,1432-1297},
   MRCLASS = {14F30 (14D24 14F10)},
  MRNUMBER = {4384193},
MRREVIEWER = {Nguy\cftil en Qu\^oc Th\'ang},
       DOI = {10.1007/s00222-021-01079-5},
       URL = {https://doi.org/10.1007/s00222-021-01079-5},
}

@Article{GelfandGraev1962,
  author     = {Gelfand, I. M. and Graev, M. I.},
  journal    = {Dokl. Akad. Nauk SSSR},
  title      = {Categories of group representations and the classification problem of irreducible representations},
  year       = {1962},
  issn       = {0002-3264},
  pages      = {757--760},
  volume     = {146},
  fjournal   = {Doklady Akademii Nauk SSSR},
  mrclass    = {20.80},
  mrnumber   = {142668},
  mrreviewer = {P.\ Constantinescu},
}

@article {dCHL,
    AUTHOR = {de Cataldo, Mark Andrea and Haines, Thomas J. and Li, Li},
     TITLE = {Frobenius semisimplicity for convolution morphisms},
   JOURNAL = {Math. Z.},
  FJOURNAL = {Mathematische Zeitschrift},
    VOLUME = {289},
      YEAR = {2018},
    NUMBER = {1-2},
     PAGES = {119--169},
      ISSN = {0025-5874,1432-1823},
   MRCLASS = {14G15 (14F05 14M15)},
  MRNUMBER = {3803785},
MRREVIEWER = {Alan\ Koch},
       DOI = {10.1007/s00209-017-1946-4},
       URL = {https://doi.org/10.1007/s00209-017-1946-4},
}

@Article{HNY,
  author     = {Heinloth, Jochen and Ng\^o, Bao-Ch\^au and Yun, Zhiwei},
  journal    = {Ann. of Math. (2)},
  title      = {Kloosterman sheaves for reductive groups},
  year       = {2013},
  issn       = {0003-486X,1939-8980},
  number     = {1},
  pages      = {241--310},
  volume     = {177},
  doi        = {10.4007/annals.2013.177.1.5},
  fjournal   = {Annals of Mathematics. Second Series},
  mrclass    = {22E57 (11L05)},
  mrnumber   = {2999041},
  mrreviewer = {K.\ Strambach},
  url        = {https://doi.org/10.4007/annals.2013.177.1.5},
}

@Article{Gelfand1970,
  author     = {Gel'fand, S. I.},
  journal    = {Math. USSR Sb.},
  title      = {Representations of the full linear group over a finite field},
  year       = {1970},
  pages      = {13--39},
  volume     = {12(1)},
  mrclass    = {20.80},
  mrnumber   = {0272916},
  mrreviewer = {J. D. Dixon},
}

@incollection {Geck17,
    AUTHOR = {Geck, Meinolf},
     TITLE = {Minuscule weights and {C}hevalley groups},
 BOOKTITLE = {Finite simple groups: thirty years of the atlas and beyond},
    SERIES = {Contemp. Math.},
    VOLUME = {694},
     PAGES = {159--176},
 PUBLISHER = {Amer. Math. Soc., Providence, RI},
      YEAR = {2017},
      ISBN = {978-1-4704-3678-0; 978-1-4704-4168-5},
   MRCLASS = {20G40 (17B45)},
  MRNUMBER = {3682597},
MRREVIEWER = {Sergey\ Sinchuk},
}

@Article{GH96,
author = "M. Geck and G. Hiss and F.
                  L{\"u}beck and G. Malle and G. Pfeiffer",
title = "{\sf CHEVIE} -- {A} system for computing and processing
                  generic character tables for finite groups of {L}ie
                  type, {W}eyl groups and {H}ecke algebras",
journal = "Appl. Algebra Engrg. Comm. Comput.",
volume = 7,
pages = "175--210",
year = 1996,
}

@Book{PiatetskiShapiro1983,
  author     = {Piatetski-Shapiro, Ilya},
  publisher  = {American Mathematical Society, Providence, RI},
  title      = {Complex representations of {${\rm GL}(2,\,K)$}\ for finite fields {$K$}},
  year       = {1983},
  isbn       = {0-8281-5019-9},
  series     = {Contemporary Mathematics},
  volume     = {16},
  mrclass    = {20G05 (10D40 22E50)},
  mrnumber   = {696772},
  mrreviewer = {Stephen\ Gelbart},
  pages      = {vii+71},
}

@InCollection{CurtisShinoda2004,
  author     = {Curtis, Charles W. and Shinoda, Ken-ichi},
  booktitle  = {Representation theory of algebraic groups and quantum groups},
  publisher  = {Math. Soc. Japan, Tokyo},
  title      = {Zeta functions and functional equations associated with the components of the {G}elfand-{G}raev representations of a finite reductive group},
  year       = {2004},
  pages      = {121--139},
  series     = {Adv. Stud. Pure Math.},
  volume     = {40},
  doi        = {10.2969/aspm/04010121},
  mrclass    = {20G05 (20G40)},
  mrnumber   = {2074592},
  mrreviewer = {Karine Sorlin},
  url        = {https://doi.org/10.2969/aspm/04010121},
}

@Article{Carter1992,
  author   = {Carter, R. W.},
  journal  = {Proc. London Math. Soc. (3)},
  title    = {Cuspidal matrix representations for {${\rm GL}_2(q)$} and {${\rm GL}_3(q)$}},
  year     = {1992},
  issn     = {0024-6115},
  number   = {3},
  pages    = {487--523},
  volume   = {64},
  doi      = {10.1112/plms/s3-64.3.487},
  fjournal = {Proceedings of the London Mathematical Society. Third Series},
  mrclass  = {20G05 (20C15 22E55)},
  mrnumber = {1152995},
  url      = {https://doi.org/10.1112/plms/s3-64.3.487},
}

@Article{DeriziotisGotsis1998,
  author     = {Deriziotis, D. I. and Gotsis, C. P.},
  journal    = {LMS J. Comput. Math.},
  title      = {The cuspidal modules of the finite general linear groups},
  year       = {1998},
  issn       = {1461-1570},
  pages      = {75--108},
  volume     = {1},
  doi        = {10.1112/S1461157000000152},
  fjournal   = {LMS Journal of Computation and Mathematics},
  mrclass    = {20C33 (20G40)},
  mrnumber   = {1642103},
  mrreviewer = {R. W. Carter},
  url        = {http://dx.doi.org/10.1112/S1461157000000152},
}

@PhdThesis{Gotsis1997,
  author = {Gotsis, C. P.},
  school = {Athens University},
  title  = {Complex representations of the general linear group {GLn}(q)},
  year   = {1997},
}

@Article{HelversenPasotto1982,
  author   = {Helversen-Pasotto, Anna},
  journal  = {Mathematische Annalen},
  title    = {Darstellungen von $\mathrm{GL}(3, {F}_q)$ und {G}außsche {S}ummen.},
  year     = {1982},
  pages    = {1-22},
  volume   = {260},
  url      = {http://eudml.org/doc/163651},
}

@Article{Tulunay2004,
  author     = {Tulunay, I.},
  journal    = {Comm. Algebra},
  title      = {Cuspidal modules as summands of a {G}el'fand-{G}raev module},
  year       = {2004},
  issn       = {0092-7872},
  number     = {4},
  pages      = {1519--1530},
  volume     = {32},
  doi        = {10.1081/AGB-120028796},
  fjournal   = {Communications in Algebra},
  mrclass    = {20G05 (20G40)},
  mrnumber   = {2100372},
  mrreviewer = {Karine Sorlin},
  url        = {https://doi.org/10.1081/AGB-120028796},
}

@Article{ShinodaTulunay2005,
  author     = {Shinoda, Kenichi and Tulunay, Ilknur},
  journal    = {J. Algebra Appl.},
  title      = {Representations of the {H}ecke algebra for {${\rm GL}_4(q)$}},
  year       = {2005},
  issn       = {0219-4988},
  number     = {6},
  pages      = {631--644},
  volume     = {4},
  doi        = {10.1142/S0219498805001459},
  fjournal   = {Journal of Algebra and its Applications},
  mrclass    = {20C08 (20C33)},
  mrnumber   = {2192148},
  mrreviewer = {Robert B\'{e}dard},
  url        = {https://doi.org/10.1142/S0219498805001459},
}

@book {DJ20,
    AUTHOR = {Digne, Fran\c cois and Michel, Jean},
     TITLE = {Representations of finite groups of {L}ie type},
    SERIES = {London Mathematical Society Student Texts},
    VOLUME = {95},
   EDITION = {Second},
 PUBLISHER = {Cambridge University Press, Cambridge},
      YEAR = {2020},
     PAGES = {vii+257},
      ISBN = {978-1-108-72262-9; 978-1-108-48148-9},
   MRCLASS = {20C33 (20D06 20G05)},
  MRNUMBER = {4211777},
}

@Article{Nien2017,
  author     = {Nien, Chufeng},
  journal    = {Finite Fields Appl.},
  title      = {{$n\times 1$} local gamma factors and {G}auss sums},
  year       = {2017},
  issn       = {1071-5797},
  pages      = {255--270},
  volume     = {46},
  doi        = {10.1016/j.ffa.2017.04.005},
  fjournal   = {Finite Fields and their Applications},
  mrclass    = {20C33 (11L05)},
  mrnumber   = {3655759},
  mrreviewer = {John T. Cullinan},
  url        = {https://doi.org/10.1016/j.ffa.2017.04.005},
}

@Article{Zelingher2023,
  author   = {Zelingher, Elad},
  journal  = {Adv. Math.},
  title    = {On values of the {B}essel function for generic representations of finite general linear groups},
  year     = {2023},
  issn     = {0001-8708,1090-2082},
  pages    = {Paper No. 109314, 47},
  volume   = {434},
  doi      = {10.1016/j.aim.2023.109314},
  fjournal = {Advances in Mathematics},
  mrclass  = {20C33 (11L05 11T24)},
  mrnumber = {4647744},
  url      = {https://doi.org/10.1016/j.aim.2023.109314},
}

@Book{Katz1988,
  author     = {Katz, Nicholas M.},
  publisher  = {Princeton University Press, Princeton, NJ},
  title      = {Gauss sums, {K}loosterman sums, and monodromy groups},
  year       = {1988},
  isbn       = {0-691-08432-7; 0-691-08433-5},
  series     = {Annals of Mathematics Studies},
  volume     = {116},
  doi        = {10.1515/9781400882120},
  mrclass    = {11G40 (11L05 11T24 14G10 14G15)},
  mrnumber   = {955052},
  mrreviewer = {P. Bayer},
  pages      = {x+246},
  url        = {https://doi.org/10.1515/9781400882120},
}

@Book{Deligne1977,
  author    = {Deligne, P.},
  publisher = {Springer-Verlag, Berlin},
  title     = {Cohomologie \'{e}tale},
  year      = {1977},
  isbn      = {3-540-08066-X; 0-387-08066-X},
  note      = {S\'{e}minaire de g\'{e}om\'{e}trie alg\'{e}brique du Bois-Marie SGA $4\frac{1}{2}$},
  series    = {Lecture Notes in Mathematics},
  volume    = {569},
  doi       = {10.1007/BFb0091526},
  mrclass   = {14F20},
  mrnumber  = {463174},
  pages     = {iv+312},
  url       = {https://doi.org/10.1007/BFb0091526},
}

@Article{Nien2014,
  author     = {Nien, Chufeng},
  journal    = {Amer. J. Math.},
  title      = {A proof of the finite field analogue of {J}acquet's conjecture},
  year       = {2014},
  issn       = {0002-9327},
  number     = {3},
  pages      = {653--674},
  volume     = {136},
  fjournal   = {American Journal of Mathematics},
  mrclass    = {20G40 (20G05)},
  mrnumber   = {3214273},
  mrreviewer = {Daniel Goldstein},
  url        = {https://doi-org/10.1353/ajm.2014.0020},
}

@MastersThesis{Roditty2010,
  author = {Roditty, Edva-Aida},
  school = {Tel Aviv University},
  title  = {On Gamma factors and {B}essel functions for representations of general linear groups over finite fields},
  year   = {2010},
}

@Article{SoudryZelingher2023,
  author   = {Soudry, David and Zelingher, Elad},
  journal  = {Essent. Number Theory},
  title    = {On gamma factors for representations of finite general linear groups},
  year     = {2023},
  issn     = {2834-4626,2834-4634},
  number   = {1},
  pages    = {45--82},
  volume   = {2},
  doi      = {10.2140/ent.2023.2.45},
  fjournal = {Essential Number Theory},
  mrclass  = {11F66 (11T24 20C33)},
  mrnumber = {4681468},
  url      = {https://doi.org/10.2140/ent.2023.2.45},
}

@Article{JacquetPiatetskiShapiroShalika1983,
  author     = {Jacquet, H. and Piatetskii-Shapiro, I. I. and Shalika, J. A.},
  journal    = {Amer. J. Math.},
  title      = {Rankin-{S}elberg convolutions},
  year       = {1983},
  issn       = {0002-9327},
  number     = {2},
  pages      = {367--464},
  volume     = {105},
  doi        = {10.2307/2374264},
  fjournal   = {American Journal of Mathematics},
  mrclass    = {11F67 (11F70 11R39 22E55)},
  mrnumber   = {701565},
  mrreviewer = {Freydoon Shahidi},
  url        = {https://doi.org/10.2307/2374264},
}

@Article{Shahidi1984,
  author     = {Shahidi, Freydoon},
  journal    = {Amer. J. Math.},
  title      = {Fourier transforms of intertwining operators and {P}lancherel measures for {${\rm GL}(n)$}},
  year       = {1984},
  issn       = {0002-9327},
  number     = {1},
  pages      = {67--111},
  volume     = {106},
  doi        = {10.2307/2374430},
  fjournal   = {American Journal of Mathematics},
  mrclass    = {22E50 (11S37)},
  mrnumber   = {729755},
  mrreviewer = {Allan J. Silberger},
  url        = {https://doi.org/10.2307/2374430},
}

@Article{Shahidi1990,
  author     = {Shahidi, Freydoon},
  journal    = {Ann. of Math. (2)},
  title      = {A proof of {L}anglands' conjecture on {P}lancherel measures; complementary series for {$p$}-adic groups},
  year       = {1990},
  issn       = {0003-486X},
  number     = {2},
  pages      = {273--330},
  volume     = {132},
  doi        = {10.2307/1971524},
  fjournal   = {Annals of Mathematics. Second Series},
  mrclass    = {11R39 (11F70 11S37 22E35 22E55)},
  mrnumber   = {1070599},
  mrreviewer = {Stephen Gelbart},
  url        = {https://doi.org/10.2307/1971524},
}

@Article{GGP2012,
  author   = {Gan, Wee Teck and Gross, Benedict H. and Prasad, Dipendra},
  journal  = {Ast\'{e}risque},
  title    = {Symplectic local root numbers, central critical {$L$} values, and restriction problems in the representation theory of classical groups},
  year     = {2012},
  issn     = {0303-1179},
  note     = {Sur les conjectures de Gross et Prasad. I},
  number   = {346},
  pages    = {1--109},
  fjournal = {Ast\'{e}risque},
  isbn     = {978-2-85629-348-5},
  mrclass  = {22E50 (11F70 11R39 22E55)},
  mrnumber = {3202556},
}

@Article{Rainbolt2002,
  author     = {Rainbolt, Julianne G.},
  journal    = {Comm. Algebra},
  title      = {The irreducible representations of the {H}ecke algebras constructed from the {G}elfand-{G}raev representations of {${\rm GL}(3,q)$} and {${\rm U}(3,q)$}},
  year       = {2002},
  issn       = {0092-7872,1532-4125},
  number     = {9},
  pages      = {4085--4103},
  volume     = {30},
  doi        = {10.1081/AGB-120013305},
  fjournal   = {Communications in Algebra},
  mrclass    = {20C08},
  mrnumber   = {1936457},
  mrreviewer = {Richard\ M.\ Green},
  url        = {https://doi.org/10.1081/AGB-120013305},
}

@Article{Rainbolt2008,
  author   = {Rainbolt, Julianne G.},
  journal  = {J. Algebra},
  title    = {Notes on the norm map between the {H}ecke algebras of the {G}elfand-{G}raev representations of {${\rm GL}(2,q^2)$} and {${\rm U}(2,q)$}},
  year     = {2008},
  issn     = {0021-8693,1090-266X},
  number   = {9},
  pages    = {3493--3511},
  volume   = {320},
  doi      = {10.1016/j.jalgebra.2008.07.025},
  fjournal = {Journal of Algebra},
  mrclass  = {20C08 (20C33)},
  mrnumber = {2455512},
  url      = {https://doi.org/10.1016/j.jalgebra.2008.07.025},
}

@Article{BreedingAllisonRainbolt2019,
  author     = {Breeding-Allison, Jeffery and Rainbolt, Julianne},
  journal    = {Comm. Algebra},
  title      = {{T}he {G}elfand-{G}raev representation of {${\rm GSp}(4,\mathbb{F}_q)$}},
  year       = {2019},
  issn       = {0092-7872,1532-4125},
  number     = {2},
  pages      = {560--584},
  volume     = {47},
  doi        = {10.1080/00927872.2018.1485228},
  fjournal   = {Communications in Algebra},
  mrclass    = {20C08 (20C33 20G05 20G40)},
  mrnumber   = {3935366},
  mrreviewer = {Olivier\ Dudas},
  url        = {https://doi.org/10.1080/00927872.2018.1485228},
}

@Article{Chang1976,
  author     = {Chang, Bomshik},
  journal    = {Comm. Algebra},
  title      = {Decomposition of {G}elfand-{G}raev characters of {${\rm GL}\sb{3}(q)$}},
  year       = {1976},
  issn       = {0092-7872},
  number     = {4},
  pages      = {375--401},
  volume     = {4},
  doi        = {10.1080/00927877608822112},
  fjournal   = {Communications in Algebra},
  mrclass    = {20G40},
  mrnumber   = {401940},
  mrreviewer = {T. V. Fossum},
  url        = {https://doi.org/10.1080/00927877608822112},
}

@Article{CassGuZelingher2026GitHub,
  author = {Robert Cass and Miao (Pam) Gu and Elad Zelingher},
  title  = {{C}ompanion {P}rograms for {K}loosterman sheaves and {B}essel functions for generic principal series representations},
  year   = {2026},
  note   = {Available at \href{https://cgz2026.github.io/}{https://cgz2026.github.io/}},
}
\end{document}